\documentclass[12pt]{article}

\usepackage{amsmath,amssymb,latexsym}
\usepackage{bbm} 
\usepackage{epsfig}
\usepackage{lscape}
\usepackage{array}

\def\bI{\boldsymbol{I}}

\def\bO{\boldsymbol{O}}

\def\binfim{\boldsymbol{-}\boldsymbol{\infty}}
\def\binfip{\boldsymbol{+}\boldsymbol{\infty}}

\def\real{\mathbb{R}}

\def\relative{\mathbb{Z}}

\def\x{\mathbbm{x}}
\def\y{\mathbbm{y}}

\def\I{\mathbb{I}}

\def\T{\mathbb{T}}

\def\B{\mathbb{B}}
\def\U{\mathbb{U}}
\def\mat1{\mathbb{I}}

\def\dzero{\mbox{{\scriptsize $\mathbb{O}$}}}  
\def\dun{\mathbbm 1}

\def\limk{\lim_{k \rightarrow +\infty}}

\newcommand { \Iletter}[1] {I\kern-0.10em #1 }

\def\bit{\begin{itemize}}
\def\eit{\end{itemize}}
\def\ben{\begin{enumerate}}
\def\een{\end{enumerate}}
\def\bde{\begin{description}}
\def\ede{\end{description}}

\def\bar{\begin{array}}
\def\ear{\end{array}}
\def\beq{\begin{equation}}
\def\eeq{\end{equation}}
\def\bfi{\begin{figure}[hbt] \begin{center}}
\def\efi{\end{center} \end{figure}}
\def\noi{\noindent}
\def\bce{\begin{center}}
\def\ece{\end{center}}
\newcommand{\proof}{{\bf Proof. }}
\newcommand{\cqfd}{\hfill $\Box$}

\newtheorem {theo} {Theorem}[section]

\newtheorem {lemm} {Lemma}[section]

\newtheorem {propo} {Proposition}[section]

\newtheorem {defi} {Definition}[section]

\newtheorem {rem} {Remark} [section]

\newtheorem {resu} {Result} [section]

\newtheorem {Assum} {Assumption} [section]

\newtheorem {numerotation} {Numerotation convention} [section]

\begin{document}        

     

\title {Substitution for minimizing/maximizing a tropical linear (fractional) programming}

\author{L. Truffet \\
  IMT Atlantique \\
  Department Automation, Production and Computer Sciences \\
  La Chantrerie, 4 rue A. Kastler
  44300 Nantes, France  \\
  mail: laurent.truffet@imt-atlantique.fr}

\maketitle 

\begin{abstract}
  Tropical polyhedra seem to play a central role in static analysis of softwares. These
  tropical geometrical objects play also a central role in parity games especially mean payoff
  games and energy games. And determining if an initial state of such game leads to win the game is known
  to be equivalent to solve a tropical linear optimization problem.
This paper mainly focus on the tropical linear minimization problem using
  a special substitution method on the tropical cone obtained by homogenization of
  the initial tropical polyhedron. But due to a particular case which can occur
  in the minimization process based on substitution we have to switch on a maximization
  problem. Nevertheless, forward-backward substitution is known to be
  strongly polynomial. The special substitution developed in this paper inherits
  the strong polynomiality of the classical substitution for linear systems. This
  special substitution must not be confused
  with the exponential execution time of the tropical Fourier-Motzkin elimination. Tropical
  fractional minimization problem with linear objective functions is also solved by tropicalizing
  the Charnes-Cooper's transformation
  of a fractional linear program into a linear program developed in the usual linear
  algebra. Let us also remark that no particular assumption is made on the polyhedron of interest.

  Finally, the substitution method is illustrated on some examples borrowed from the litterature.
  \end{abstract}

\noi
    {\bf Keywords}: $(\max,+)$-algebra, tropical linear programming, fractional programming, algorithmic complexity. \\
    \noi
    MSC: 15A80, 16Y60, 90C32, 03D15. \\

    \noi
    NOTA:

    This new version correct a bad logical mistake which have been done in the previous versions of
    that preprint for the minimization problem of a tropical linear program. Indeed, the minimization
    problem deals with inequalities describing tropical linear half-spaces involved in the problem which are
    of the form $\max(a,b) \geq c$. Such inequalities are logically equivalent to the
    disjunction: $(a \geq c) \mbox{ or } (b \geq c)$. Unfortunately, this disjunction has been
    treated as if it was the conjunction $(a \geq c) \mbox{ and } (b \geq c)$. Thus,
    from the logical point of view the minimization problem
    deals with conjunction of a set of disjunctions (ie. intersection of half-spaces). However, let us
    stress that this logical error
    does not question the general principle of the substitution method.

    Note also that the maximization problem
    deals with inequalities describing the same tropical linear half-spaces involved in the problem which
    are of the form $\max(a',b') \leq c'$. Theses inequalities are logically equivalent to the
    conjunction: $(a' \leq c') \mbox{ and } (b' \leq c')$. Thus, the maximization problem only deals with
    conjunction of a set of conjunctions (intersection of half-spaces) which has already been
    treated in the Appendix A in the previous versions of that preprint.





\section{Main notations and definitions}
\label{secNotations}

For every $k \leq k'$, $k,k' \in \relative$ we define the 
discrete interval $[| k,k' |]:=\{k, k+1, \ldots, k'-1,k'\}$. 
The inclusion of sets is denoted $\subseteq$ and the strict inclusion 
of set is denoted $\subset$. \\

Let us introduce the following
notations and definitions for the $(\max,+)$-algebra. The reader is invinted to read eg.
\cite{kn:Bac-cooq} to find more details on $(\max,+)$-algebra and idempotent semi-rings.

The $(\max,+)$-algebra denotes the set $\real_{\dzero}:= \real \cup \{-\infty \}$
equipped the ``addition'': $(a,b) \mapsto \max(a,b)$ and the ``multiplication''
$(a,b) \mapsto a + b$.

\noi
In the sequel we will use the following notations: $a \oplus b := \max(a,b)$,
$a \otimes b   := a+ b$ (sometimes $a \otimes b$ could be denoted  $ab$). 
  $\dun :=0$ is the neutral element
  for $\otimes$, $\dzero :=-\infty$ is the neutral element for $\oplus$, $\dzero$ is also the absorbing element
  of $\otimes$.
\noi
  We will use the following power notation:

  \[
  a^{\otimes (b)}:= a \times b \mbox{ ($\times$: usual multiplication)}.
  \]
  Note that in particular we have: $a^{\otimes (0)}=\dun$.

  The product $a \otimes b^{\otimes (-1)}$ will also be
  denoted: $a / b$. Because division is required in this paper we have: $\dzero^{\otimes (-1)}=+\infty$.
  Thus, we work in the complete dioid by adding the top element $+\infty$:
\begin{subequations}
  \begin{equation}
(\overline{\real}_{\dzero}, \oplus, \otimes)
    \end{equation}
with:
  \begin{equation}
\overline{\real}_{\dzero}:= \real \cup \{-\infty, +\infty\},
  \end{equation}
  and the operations $\oplus$ and $\otimes$ are extended as follows:
  \begin{equation}
\forall a \in \overline{\real}_{\dzero},  \forall b \neq \dzero, \; a \oplus +\infty = +\infty, \; b \otimes +\infty =+\infty, \; \dzero \otimes +\infty=\dzero.
    \end{equation}
    \end{subequations}
 \noi
   $\overline{\real}_{\dzero}$ is equipped with the natural order $\leq$ defined by:
  \[
a \leq b \Leftrightarrow a \oplus b = b. 
  \]
\noi

  \noi
  For any matrix $A$, $a_{i,.}$ denotes the $i$th row of $A$, $a_{.,j}$ denotes the
  $j$th column of $A$ and $a_{i,j}$ denotes the entry $(i,j)$ of $A$. \\

  The scalar operations are generalized to matrices as follows:

  \bit
\item addition of two matrices: $(a_{i,j}) \oplus (b_{i,j}) := (a_{i,j} \oplus b_{i,j})$ 
\item product of two matrices: the entry $(i,j)$ of the matrix $C=A \otimes B$ is
  defined by: \\
  $c_{i,j}:= \oplus_{k} a_{i,k} \otimes b_{k,j}$ (=$\max_{k}(a_{i,k}+b_{k,j})$).

\item comparaison of two matrices: $A \leq B$ means $\forall i,j: a_{i,j} \leq b_{i,j}$.

  \item all vectors are column vectors. And the transpose operator is denoted $(\cdot)^{\intercal}$.
  \item for every $m \times n$-matrix $A$ we define the submatrix of $A$ denoted $A_{IJ}$ by:
    \begin{equation}
      \label{eqSubMatNotation}
A_{IJ}:= (a_{i,j})_{i \in I, j \in J}.
    \end{equation}
    Where $I \subseteq [|1,m|]$ and $J \subseteq [|1,n|]$.
  \eit

  We will use the bold symbols for the following particular vectors and matrices.
  
  \bit
\item The boldsymbol $\binfim$ (resp. $\binfip$) denotes the
  column vector which all components are $-\infty$ (resp. $+\infty$). The number of components
  of the vector is determined by the context.
\item The matrix $\bI_{n}$ denotes the $n \times n$-identity matrix. Its diagonal
  entries are all $\dun$ and its off-diagonal entries are all $\dzero$.
\item $\bO_{m,n}$ will denote the $m \times n$-null matrix (ie, all entries are $\dzero$). 
  When $m=n$ the the square null matrix $\bO_{n,n}$ is denoted $\bO_{n}$.
  
  \eit
    
To any $n \times n$-matrix $A$ we can associate a valued graph $G(A)$ with valuation $v$ such that
 the set of the vertices of $G(A)$ is $\{1, \ldots, n\}$. The valuation is defined by $v_{i,j}=a_{j,i}$ for all
 vertices $i,j$. We say that an arc $i \rightarrow j :=(i,j)$ exists if $v_{i,j} \neq \dzero$. A path $p$ of lenght
 $k \geq 1$ is
 a series of $k$ consecutive arcs which has the form: $p=i_{0} \rightarrow i_{1} \rightarrow \cdots i_{k-1} \rightarrow
 i_{k}$. And the weight of $p$ is defined by: $w(p):= v_{i_{0}, i_{1}} \otimes \cdots \otimes v_{i_{k-1}, i_{k}}=
   a_{i_{1}, i_{0}} \otimes \cdots \otimes a_{i_{k}, i_{k-1}}$. An elementary circuit of lenght $k \geq 1$ is a
   special path such that $i_{k}=i_{0}$ and $\forall l=1, \ldots, k-1$: $i_{l}$ appears only once.

Let $q, s$ be two $n$-dimensional vectors and $x=(x_{j})_{j=1}^{n}$. Then, to any inequality of the form:

\begin{subequations}
  \begin{equation}
    \label{a+a-ineq}
q^{\intercal} \otimes x \geq s^{\intercal} \otimes x 
  \end{equation}
  Noticing that (\ref{a+a-ineq}) is expressed in usual notations as 
  $(\max(q_{1} + x_{1}, \ldots, q_{n} + x_{n}) \geq (\max(s_{1} + x_{1}, \ldots, s_{n} + x_{n})$
  one can associate the following equivalent disjunction recalling that
  $\max(a,b) \geq c \Leftrightarrow (a \geq c) \mbox{ or } (b \geq c)$:

  \begin{equation}
    \label{ORa+a-ineq}
\textsf{OR}_{j=1}^{n} \{q_{j} \otimes x_{j} \geq s^{\intercal} \otimes x\}.
    \end{equation}
  And one can associate the following equivalent conjonction recalling that
  $\max(a,b) \leq c \Leftrightarrow (a \leq c) \mbox{ and } (b \leq c)$:

  \begin{equation}
    \label{ANDa+a-ineq}
\textsf{AND}_{j=1}^{n} \{s_{j} \otimes x_{j} \leq q^{\intercal} \otimes x \}.
    \end{equation}

\end{subequations}

\section{Introduction} 
\label{secIntro}

\subsection{Problem description}

Let $A$ and $C$ be two $m \times n$-matrices and $b$ and $d$ be two $m$-dimensional column vectors.
In this paper we aim to solve by a special
substitution method the following max-plus linear programming problem called {\em
  minimum problem research} (\textsf{mpr}): 

\begin{subequations}
\begin{equation}
\textsf{min}\{z=c^{\intercal} \otimes x, x \in \mathcal{P}(A,b,C,d) \cup \{\boldsymbol{ \pm\infty}\}\},
  \end{equation}
where $\mathcal{P}(A,b,C,d)$ denotes the max-plus polyhedron defined by:
\begin{equation}
  \label{eqdefPAbCd}
\mathcal{P}(A,b,C,d):=\{x \in \overline{\real}_{\dzero}^{n}: A \otimes x \oplus b \geq C \otimes x \oplus d\},
\end{equation}

\end{subequations}

But instead of considering the polyhedron we will develop substitution
method on the max-plus cone $\mathcal{C}(A,b,C,d)$ obtained by introducing
$h$ as the homogenization variable, ie.:

\begin{subequations}
\begin{equation}
  \mathcal{C}(A,b,C,d):=\{(x,h) \in \overline{\real}_{\dzero}^{n+1}: A \otimes x \oplus b \otimes h \geq C \otimes x \oplus d \otimes h\}\cup \{\boldsymbol{\pm \infty}\}.
\end{equation}

The linear system of inequalities $A \otimes x \oplus b \otimes h \geq C \otimes x \oplus d \otimes h$ is supposed
to verify the following condition:

\begin{equation}
  \label{maximalitycondition}
  \bar{l}
  \forall i \in [|1,m|] \\
  \{i': i' \neq i \mbox{ and } \exists \alpha_{i'} \in  \real \mbox{ s.t. } a_{i,.}
  \otimes x  \oplus b_{i} \otimes h \geq 
  \alpha_{i'} \otimes (c_{i',.} \otimes x \oplus  d_{i'} \otimes h) \} = \emptyset.
  \ear
  \end{equation}

\end{subequations}

\begin{rem}
  The condition (\ref{maximalitycondition}) can be assumed to be always true because
  $f \geq g_{1}$ and $f \geq g_{2}$ is equivalent to $f \geq g_{1} \oplus g_{2}$. 
  \end{rem}

To formulate the minimization problem $\textsf{mpr}$ first we borrow the Fourier's trick
(see eg. \cite{kn:Juhel2025}, \cite{kn:Williams86}) used
in the usual linear algebra and applying it to the $(\max,+)$-algebra. It comes
that $\min(z=c^{\intercal} \otimes x \oplus c_{h} \otimes h)$ can be replaced by:
$\min(z) \mbox{ and } z \geq c^{\intercal} \otimes x \oplus c_{h} \otimes h$. And then, we will solve the following
minimization problem $\textsf{mpr}$ defined as:

\begin{subequations}
  \begin{equation}
    \label{minz}
\bar{l}
\min(z)
\ear
    \end{equation}
 such that: 
 \begin{equation}
   \label{domValxh}
(x,h) \in \overline{\real}_{\dzero}^{n+1},
   \end{equation}
and
 \begin{equation}
   \label{mprContraintes}
   \bar{l}
    z \geq c^{\intercal} \otimes x \oplus c_{h} \otimes h \\
    A \otimes x \oplus b \otimes h \geq C \otimes x \oplus d \otimes h.
  \ear
\end{equation}
 Where the homogenization variable $h$ satisfies the following condition:
 \begin{equation}
   \label{cond-h-different-infini}
h > \dzero.
 \end{equation}
And
\begin{equation}
  \label{zhnonsubsti}
\mbox{$z$ and $h$ are not substituable}.
\end{equation}
Finally, among the variables of the problem $z, x_{1}, \ldots, x_{n}$ we have:
\begin{equation}
  \label{zprio}
  \mbox{$z$ has the highest priority because bounded by the cost function}.
  \end{equation}
\end{subequations}

Sometimes it will be useful to denote by $\textsf{mpr}(A,b,C,d,c,c_{h},z,x,h,(m,n))$
the \textsf{mpr} defined (\ref{minz})-(\ref{domValxh})-(\ref{mprContraintes})-(\ref{cond-h-different-infini})-(\ref{zhnonsubsti})-(\ref{zprio}).

\subsection{Motivations}

Motivations for studying tropical linear programms have been widely explained in several previous works by several
authors. In \cite{kn:BA08} author motivate tropical linear optimization by aiming to solve
the problem of the synchronization of multiprocessors interactive systems. In \cite{kn:GKS012} and \cite{kn:GMH014}
authors indicate further motivations to solve not only tropical linear programms but also
fractional tropical linear programms (see \S~\ref{seclinfraction} for the description of such problem). One can
mention eg.: static analysis of softwares, finding winning strategy in Mean Payoff Game.

\subsection{Organization of the paper}
Section~\ref{secPrel} deals with the necessary materials to solve the minimization problem \textsf{mpr}.
The substitution method is described in Section~\ref{secmin}. \\

The reader must be aware that the
substitution must not be confused with the tropical Fourier-Motzkin elimination procedure
which has an exponential execution time \cite{allamigeon:hal-01087367}. \\

The substitution is based on
the prametrized research domain (\ref{eqRxh})-(\ref{eqmudominelambda}) for the $x_{j}$'s and the
homogenization variable $h$. Dominating variable is defined in Definition~\ref{defHierarchyxj}. And the
hierarchy $\preceq_{var}$ between variables $x_{j}$ (\ref{eqOrdreVar}) characterizes the fact
that the value domain of the dominating variables has non trivial upper and lower bounds. The partition
of the $(x,h)$-linear functions into $h$-bounded and $h$-unbounded functions is defined in Definition~\ref{defhBhU}. 
The hierarchy $\preceq_{fct}$ between $(x,h)$-linear functions (\ref{eqOrdrefct}) is based on the
obvious property that value domain of a $h$-bounded $(x,h)$-linear function
$\alpha^{\intercal} \otimes x \oplus \beta \otimes h$ only depends on $\beta$ and
is strictly included in the value domain of the
function $\alpha^{\intercal} \otimes x$ (see (\ref{eqfRdomaine})-(\ref{eqDomainhB})-(\ref{eqDominclus})).
The Theorem~\ref{theoNTZmin} of \S~\ref{subsubChoixvar} characterizes the relevant variable
$x_{j^{*}}$ which can be substituted. And the Theorem~\ref{thmchattaignable} of \S~\ref{subsubChoixvar}
ensures the optimality of the result. The update of the \textsf{mpr} is described in
\S~\ref{subsecNewcost+Newconepmrp}. Section~\ref{seclinfraction} deals with the fractional
linear programming. The method is illustrated on numerical examples in Section~\ref{secExNum}.
Finally, we conclude in Section~\ref{secConcl} where the complexity of our substitution method
(about $\mathcal{O}(mn^{3})$) is more discussed and detailed. We add some remarks on this work vs other works. 
Finally, the maximization problem is presented in appendix~\ref{secMPR}.

\subsection{The switch to a maximizing problem can occur}
\label{subswitch}
The appendix~\ref{secMPR} dealing with maximization problem is needed because the
following situation can occur in the substitution process for \textsf{mpr}. The reader has noticed that
Fourier's trick induces that the cost function is a lower bound of the variable $z$. Thus, because all linear
functions are non-decreasing the research of lower bounds on the $x_{j}$'s is the priority. But in the substitution process
we can have the following situation:
\bit
\item the cost function only depends on the homogenization variable $h$, ie. $z \geq c_{h}\otimes h$,

\item not all the variables have been substituted and they have no lower bounds but upper bounds
  exist at least for one variable,

\item the optimality is not ensured, ie. Theorem~\ref{thmchattaignable} of \S~\ref{subsubChoixvar}
  does not apply.

\eit

Noticing that minimazing $c_{h}\otimes h$ or maximazing $c_{h}\otimes h$ over the
set of the remaining variables does not make differences
one can switch to the maximization problem with $z \leq c_{h}\otimes h$ (Fourier's trick for
$\max$) and compute upper bounds on the remaining $x_{j}$'s of the problem.

Dually, the switch from maximization problem to minimization problem can occur.

\section{Preliminary results}
\label{secPrel}
Let us consider two $m \times n$-matrices $L$ and $W$. And let us consider the
cone $\mathcal{C}(L,W):=\{x: L \otimes x \geq W \otimes x\}$. To this
cone we associate the function $\textsf{setrowtozero}(L,W)$ defined by:

  \begin{equation}
    \label{ligneazero}
    \bar{lll}
    \textsf{setrowtozero}(L,W) & := & \mbox{ For $i=1$ to $m$ do} \\
    \mbox{ } & \mbox{} & \mbox{ if $l_{i,.} \geq w_{i,.}$ then $l_{i,.}:=\dzero,  w_{i,.}:=\dzero$}
    \ear
  \end{equation}

\subsection{Elementary results from interval arithmetic in dioids}
\label{subintervalCalc}
In this subsection we present elementary materials dealing with
interval arithmetic in the $(\max,+)$-algebra. The interval arithmetic on dioids
has been already used in the context of automatic-control problem (see
eg. \cite{kn:LitSob01}, \cite{LHOMMEAU20041923}). An interval
is a set denoted $I:=[u,v]$ where $u,v \in \overline{\real}_{\dzero}$ and
$u \leq v$. We define hereafter the elementary needed in this paper:

\bit
\item Addition of two intervals. Let $I_{i}=[u_{i},v_{i}]$, $i=1,2$ be two intervals, then the
  addition of $I_{1}$ and $I_{2}$ is denoted $I_{1} \oplus I_{2}$ and defined by:

  \begin{equation}
[u_{1},v_{1}] \oplus [u_{2},v_{2}]:=[u_{1} \oplus u_{2}, v_{1} \oplus v_{2}].
    \end{equation}

\item Scalar multiplication. Let $I=[u,v]$ be an interval and let $\alpha \in \overline{\real}_{\dzero}$. Then,
  the scalar multiplication of $I$ by the scalar $\alpha$ is denoted $\alpha \otimes I$ and defined
  by:

  \begin{equation}
\alpha \otimes [u,v]:= [\alpha \otimes u, \alpha \otimes v].
    \end{equation}

\item And the linear combination of $k$ intervals is the following interval defined as:

  \begin{equation}
    \label{eqcomblininterval}
\oplus_{i=1}^{k} \alpha_{i} \otimes [u_{i}, v_{i}]:= [\oplus_{i=1}^{k} \alpha_{i} \otimes u_{i}, \oplus_{i=1}^{k} \alpha_{i} \otimes v_{i}].
  \end{equation}
  Where $\alpha_{i} \in \overline{\real}_{\dzero}$, $i=1, \ldots,k$.
  
\eit

We will have to use results dealing with particular intervals. These particular intervals are defined
as follows:

\begin{equation}
  \label{defIntervalu+infini}
\mathcal{I}:=\{[u, +\infty], u \in \overline{\real}_{\dzero}\}.
  \end{equation}

Let us point out that the familly of intervals $\mathcal{I}$ is obviously stable by
$(\max,+)$-linear combinations.

\begin{propo}
  \label{propMinInterval}
  Let $[u_{l}, +\infty]$, $l=1, \ldots,k$, $k \geq 2$ be $k$ intervals of $\mathcal{I}$ with
  $u_{1} \leq \ldots \leq u_{k}$. Then, we have the following series of intervals inclusion:
  \begin{equation}
[u_{k}, +\infty] \subseteq [u_{k-1}, +\infty] \subseteq \cdots \subseteq [u_{1},+\infty].
    \end{equation}
  \end{propo}

If $\textsf{Min}$ (resp. $\textsf{Max}$) denotes the minimum (resp. maximum)
in the sense of interval inclusion then, from Proposition~\ref{propMinInterval}
we have the following noticeable interval equalities:

\begin{equation}
  \label{eq-intersection=Min-interval}
\cap_{l=1}^{k} [u_{l}, +\infty]=[u_{k},+\infty] = \textsf{Min}\{[u_{l}, +\infty], l=1, \ldots,k\}, 
  \end{equation}
and
\begin{equation}
  \label{eq-union=Max-interval}
\cup_{l=1}^{k} [u_{l}, +\infty]=[u_{1},+\infty] = \textsf{Max}\{[u_{l}, +\infty], l=1, \ldots,k\}, 
  \end{equation}

   Let $A$ be a $n \times n$-matrix, $b$ a $n$-dimensional vector and $x=(x_{i})_{i=1}^{n}$. We have the
   following well-known result. 
   
   \begin{resu}[Sub and super fix point]
     \label{lemSAT}
     Let us consider the following sets defined by:
     \begin{subequations}
  \begin{equation}
    \label{masterineq}
\mathcal{S}^{\leq}(A,b):=\{x: x \leq A \otimes x \oplus b \}. 
  \end{equation}
  and
\begin{equation}
    \label{masterineq2}
\mathcal{S}^{\geq}(A,b):=\{x: x \geq A \otimes x \oplus b \}. 
  \end{equation}  

\end{subequations}
  
  If the following condition holds:
  
    \begin{equation}
      \label{condAinfini=O}
       \limk A^{\otimes (k)} = \bO_{n} \mbox{ ($\bO_{n}$: $n \times n$-null matrix)}
\end{equation}
    then the greatest element of the set $\mathcal{S}^{\leq}(A,b)$
    (resp. the smallest element of the set $\mathcal{S}^{\geq}(A,b)$) is:

    \[
x= A^{*} \otimes b,
\]
where $A^{*}:= \bI_{n} \oplus A \oplus A^{\otimes (2)} \oplus  A^{\otimes (3)} \oplus  \cdots $ is the
infinite ``sum'' of the powers of the matrix
$A$ known as the Kleene star of the matrix $A$. And
under condition (\ref{condAinfini=O}) we have: $A^{*}= \bI_{n} \oplus A \oplus A^{\otimes (2)} \oplus \cdots \oplus A^{\otimes (n-1)}$.

This greatest element is the solution of the $(\max,+)$-linear equation: $x= A \otimes x \oplus b$
which is obtained by the saturation of the inequality: $x \leq  A \otimes x \oplus b$.
     \end{resu}

   There exist different conditions such that (\ref{condAinfini=O}) is true. Notably the one
   which states that all elementary circuits of the valued graph $G(A)$ associated with the matrix $A$
 have weight $ < \dun$.
  
 We specify the previous Result~\ref{lemSAT} in the following case.

 \begin{propo}[Valid inequality]
   \label{propineqValid}
   Let $a$ be a scalar, $a \in \real_{\dzero}$. Let $v$ be a $n$-dimensional vector. And let $y=(y_{i})_{i=1}^{n}$ be a
   $n$-dimensional column vector of variables in $\real_{\dzero}$. Then, an inequality of the form:
   \begin{equation}
     \label{eqvalid}
a \otimes y_{j} \lesseqgtr v^{\intercal} \otimes y, 
     \end{equation}
   is said to be {\em valid} iff the following condition is fulfilled:
   \begin{equation}
     \label{condeqvalid}
     \mbox{$a \neq \dzero$ and $a > v_{j}$}.
   \end{equation}
   And in this case the inequality (\ref{eqvalid}) is equivalent to:

   \begin{subequations}
     \begin{equation}
       \label{eqvalid2}
y_{j} \lesseqgtr a^{\otimes (-1)} \otimes \overline{v}^{\intercal} \otimes y,
     \end{equation}
   where $\overline{v}$ is the $n$-dimensional vector defined by:

   \begin{equation}
     \label{defvbar}
     \forall j', \overline{v}_{j'} := \left\{\bar{ll} \dzero & \mbox{ if $j'=j$} \\
                                                     v_{j'} & \mbox{ if $j' \neq j$.}

     \ear \right.
     \end{equation}
   \end{subequations}
 \end{propo}
 \proof W.l.o.g we can assume $j=n$ in the inequality (\ref{eqvalid}). And the inequality (\ref{eqvalid}) is equivalent to

 \[
y \lesseqgtr V \otimes y,
 \]
 where $V$ is the $n \times n$-matrix defined by: 

 \[
 V:= \left(\bar{cc}
 \bI_{n-1} & \bO_{n-1,1} \\
 D & C
 \ear \right), \; D=\left(\bar{ccc} v_{1} & \cdots & v_{n-1}
 \ear \right), \; C=(a^{\otimes (-1)} \otimes v_{n}).
 \]
 Now we express $V \otimes y$ as follows:
 \[
V \otimes y = A \otimes y \oplus B \otimes u,
\]
with :
\[
A:= \left(\bar{cc}
 \bO_{n-1,n-1} & \bO_{n-1,1} \\
 \bO_{1,n-1} & C
 \ear \right),\; B:=\left(\bar{c} \bI_{n-1} \\
 D
 \ear \right), \; u=(y_{i})_{i=1}^{n-1}.
\]
And clearly by Result~\ref{lemSAT} the set $\{y: y \lesseqgtr A \otimes y \oplus B \otimes u\}$ admits
a smallest (or greatest) element iff condition (\ref{condAinfini=O}) is verified. Here, the condition
(\ref{condAinfini=O}) is equivalent to $\limk C^{\otimes (k)} = \bO_{1,1}$ which is equivalent
to $a^{\otimes (-1)} \otimes v_{n} < \dun$ because the matrix $C$ has only one elementary
circuit $n \rightarrow n$ of weight $a^{\otimes (-1)} \otimes v_{n}$. To conclude we just have to note that
by defining $\overline{v}^{\intercal}:= (v_{1}, \ldots, v_{n-1}, \dzero)$ the inequality
(\ref{eqvalid2}) is verified. \cqfd

\section{The minimization problem $\textsf{mpr}$}
\label{secmin}
At the step $0$ of the substitution the cone associated with the constraints system (\ref{mprContraintes}) of the $\textsf{mpr}$ problem
is denoted $\mathcal{C}(A^{\; +}, A^{ \; -})^{[0]}$ and is defined by:

\begin{subequations}
\begin{equation}
  \mathcal{C}(A^{+}, A^{-})^{[0]}:=\{w \in \overline{\real}_{\dzero}^{n+2}: A^{+ [0]} \otimes w \geq
  A^{-[0]} \otimes w\}
  \end{equation}
where:

\begin{equation}
  (A^{+}, A^{-}, w)^{[0]}:=\left(\bar{ccc}
  \dun & \boldsymbol{-\infty}^{\intercal} & \dzero \\
  \boldsymbol{-\infty} & A & b
  \ear\right), \; \left(\bar{ccc}
  \dzero & c^{\intercal} & c_{h} \\
  \boldsymbol{-\infty} & C & d
  \ear\right), \; \left(\bar{c}
z \\ x \\ h
  \ear\right).
  \end{equation}
  \end{subequations}

\noi

We use the following conventions. 

\begin{numerotation}
  \label{numMatcol}
  The rows of the matrices
$A^{+}, A^{-}$ are numbered from $0$ to $m$. The columns
  of the matrices $A^{+}, A^{-}$ are indexed by $j$, $j$ varying from $j=0$ to $j=n+1$.
  But sometimes it will be useful to use the variables $z,x_{1}, \ldots, x_{n},h$. The
  components of the vector $w$ are numbered from $0$ to $n+1$.
\end{numerotation}

The cost function at step $0$ of the substitution is denoted and defined by: $\textsf{cost}^{[0]}(x,h):= c^{\intercal} \otimes x \oplus c_{h} \otimes h$,
with $c_{h}=\dzero$. \\

The set of the stored linear equalities is denoted $\mathcal{L}^{[k]}$ at each step $k$ of the
substitution. And for $k=0$ we have $\mathcal{L}^{[0]}= \emptyset$. \\

The vector $x$ is called the vector of the remaining variables of the \textsf{mpr}. Of course at
step $0$ of the method: $x=(x_{j})_{j=1}^{n}$. To $x$ we associate the set of the remaining variables
denoted $\x$. And of course at step $0$ of the method: $\x^{[0]}=\{x_{1}, \ldots, x_{n}\}$. We have the
following equivalence:

\begin{equation}
x_{j} \notin \x \Leftrightarrow x_{j}=\dzero \mbox{ in the vector of remaining variables $x$}.
  \end{equation}

Because $x=\boldsymbol{+\infty} \in \mathcal{P}(A,b,C,d) \cup \{\boldsymbol{\pm \infty}\}$ and because
$\min = -\infty$ is an acceptable answer for the minimization problem \textsf{mpr} one make the
following assumption
for the parametrized reseach of a minimum reached by $x$.

\begin{Assum}[Parametrized reseach domain for the variables $x_{j}$]
  \label{AssumedomaineMINxj}
  For the research of a minimum for $z$ we can assume that $\forall x_{j} \in \x$: $x_{j} \in [\lambda, +\infty]$
  for some arbitrary $\dzero \leq \lambda \leq +\infty$. 
  \end{Assum}

The homogenization variable $h$ is not
substituable (recall). And the general scheme of the substitution for the
$\textsf{mpr}$ problem is to arrive after the $n$ substitutions of the
variables $x_{j}$ at the following situation:

\begin{equation}
  \bar{l}
  z \geq c_{h}^{[n]} \otimes h \\
  b^{[n]} \otimes h \geq d^{[n]} \otimes h. 
  \ear
  \end{equation}
Where $b^{[n]}, d^{[n]}$ are two $m$-dimensional column vectors. We then have the following cases.

\bit
\item Case $1$: $b^{[n]} \geq d^{[n]}$. In this case $h$ can be set to $\dun$ and $\min(z)=c_{h}^{[n]}$. 
\item Case $2$: $b^{[n]} \ngeq d^{[n]}$. In this case $h=+\infty$ is the only possible value for $h$ because
  $h > \dzero$ (see condition (\ref{cond-h-different-infini})) and
  $\min(z)=+\infty$.
\eit
This discussion leads to assume:

\begin{Assum}[Parametrized domain for $h$]
  \label{AssumedomaineMINh}
  For the $\textsf{mpr}$ problem we assume that the homogenization variable $h$ has a parametrized interval
  $[\mu, +\infty]$ for some arbitrary $\dzero < \mu \leq +\infty$. 
\end{Assum}

In conclusion, the parametrized domain of research of a minimum is de noted $R_{\lambda \mu}(x,h)$ and defined by:

\begin{subequations}
\begin{equation}
  \label{eqRxh}
R_{\lambda \mu}(x,h):= \times_{x_{j} \in \x} [\lambda, +\infty] \times [\mu, +\infty].
  \end{equation}
Where because $h$ must be $ > \dzero$ we can assume that the following condition holds:

\begin{equation}
  \label{eqmudominelambda}
  \forall \theta, \forall \beta \neq \dzero, \;  \exists \lambda, \mu: \;
  \theta \otimes \lambda < \beta \otimes \mu.
\end{equation}
\end{subequations}

\subsection{Study of the cone $\mathcal{C}(i):=\{z \geq c^{\intercal} \otimes x \oplus c_{h} \otimes h, \; a^{+}_{i,.} \otimes w
  \geq a^{-}_{i,.} \otimes w\}$}
\label{coneMinCi}

Based on the result of Proposition~\ref{propineqValid}
$\forall i=1, \ldots,m$, $\forall j=1, \ldots, n$ such that the following condition is
satisfied:

\begin{equation}
  \label{condExistgeq}
  \mbox{$a^{+}_{i,j} \neq \dzero$ and
    $a^{+}_{i,j} > a^{-}_{i,j}$}
\end{equation}

we define the
following valid inequality:

\begin{equation}
  \label{ineqImin}
I_{\geq}(x_{j},i):=\{a^{+}_{i,j} \otimes x_{j} \geq a^{-}_{i,.} \otimes w\}.
  \end{equation}
Under condition (\ref{condExistgeq}) the inequality $I_{\geq}(x_{j},i)$ can be rewritten as:

    \begin{equation}
I_{\geq}(x_{j},i)=\{x_{j} \geq f^{\geq}_{ij}(x,h)\}.
      \end{equation}

Where the $(x,h)$-linear function $f^{\geq}_{ij}(x,h)$ is defined by:

    \begin{subequations}
      \begin{equation}
        \label{deffij}
f^{\geq}_{ij}(x,h):= \ell_{ij}(x) \oplus r_{ij} \otimes h, 
    \end{equation}
      with $\ell_{ij}$ is the following $x$-linear function which does not
      depend on $x_{j}$:
    \begin{equation}
      \label{deflij}
\ell_{ij}(x):= v^{\intercal}_{ij} \otimes x,
    \end{equation}
    where $v_{ij}$ denotes the  $n$-dimensional vector such that:
    \begin{equation}
      \label{eqdefuij}
      \forall j'=1, \ldots, n: \; v_{ij,j'}=\left\{\bar{cc}
      (a^{+}_{i,j})^{\otimes (-1)} \otimes a^{-}_{i,j'} & \mbox{if $j' \neq j$} \\
      \dzero & \mbox{if $j'=j$}
      \ear \right.
    \end{equation}
   and the scalar $r_{ij}$ is defined by:
    \begin{equation}
      \label{eqdefrij}
r_{ij}:= (a^{+}_{i,j})^{\otimes (-1)} \otimes d_{i}.
      \end{equation}
    \end{subequations}

    Assume that $I_{\geq}(x_{j},i)$ is valid. Then, by replacing
    $x_{j}$ by $f^{\geq}_{ij}(x,h)$ in the cost function $c^{\intercal} \otimes x \oplus c_{h} \otimes h$ we define the following function:

    \begin{equation}
      \label{fzij}
f_{z,ij}(x,h):= \oplus_{j' \neq j} c_{j'} \otimes x_{j'} \oplus c_{j} \otimes f^{\geq}_{ij}(x,h) \oplus c_{h} \otimes h.
      \end{equation}
    By replacing $f^{\geq}_{ij}(x,h)$ in (\ref{fzij}) we obtain the new expression for $f_{z,ij}(x,h)$:

    \begin{equation}
      \label{fzij2}
f_{z,ij}(x,h):= \oplus_{j' \neq j} (c_{j'} \oplus c_{j} \otimes v_{ij,j'}) \otimes x_{j'} \oplus (c_{h} \oplus c_{j} \otimes r_{ij}) \otimes h.
      \end{equation}

The next Theorem will be useful to justify the substitution method developed in this paper.

    \begin{theo}[Saturation of an inequality]
      \label{theoMinCi}
      For all $x_{j}$ such that $I_{\geq}(x_{j},i)$ is valid in the sense of Proposition~\ref{propineqValid} we have:

      \begin{equation}
        \label{infimumzxj}
        \inf_{\{x: \; x_{j} \geq f^{\geq}_{ij}(x,h)\}} \left(\bar{c} c^{\intercal} \otimes x \oplus c_{h} \otimes h \\
        x_{j} \ear \right) \; = \left(\bar{c} f_{z,ij}(x,h) \\
        f^{\geq}_{ij}(x,h)
        \ear \right).
        \end{equation}
      And the infimum is reached at $x$ such that we have:
      \begin{equation}
        \label{substitxj}
x_{j} = f^{\geq}_{ij}(x,h).
        \end{equation}
      
    \end{theo}
    \proof We have the following implication:
    \[
    x_{j} \geq f^{\geq}_{ij}(x,h) \Rightarrow \forall \alpha \in \real_{\dzero}, \; \alpha \otimes x_{j} \geq \alpha \otimes
    f^{\geq}_{ij}(x,h). 
    \]
    In particular, the above implication is true for $\alpha = c_{j}$. \\
All the $(x,h)$-linear functions are non-decreasing. Noticing that the cost function:
$(x,h) \mapsto c^{\intercal} \otimes x \oplus c_{h} \otimes h$ is a $(x,h)$-linear function the
result is proved.  \cqfd

\subsection{How to choose the next variable for substitution in \textsf{mpr}}

\subsubsection{Classification of the variables and the linear functions of \textsf{mpr}}

    Based on the result of Proposition~\ref{propineqValid},
    $\forall i=1, \ldots,m$, $\forall j=1, \ldots, n$ such that the following condition is
satisfied:

\begin{equation}
  \label{condExistleq}
\mbox{$a^{-}_{i,j} \neq \dzero$ and
  $a^{+}_{i,j} < a^{-}_{i,j}$}
\end{equation}

we define the following valid inequality:

\begin{equation}
  \label{ineqIminmin}
I_{\leq}(x_{j},i):=\{a^{-}_{i,j} \otimes x_{j} \leq a^{+}_{i,.} \otimes w\}.
  \end{equation}
Under condition (\ref{condExistleq}) the result of Proposition~\ref{propineqValid} allows us to
remove $x_{j}$ from the expression $a^{+}_{i,.} \otimes w$. This new $(x,h)$-linear function
which does not depend on $x_{j}$ is denoted $f_{ij}^{\leq}(x,h)$. The expression is not given here because
the only information to know in the $\textsf{mpr}$ problem is that this function is well defined.

We use the {\bf convention} that $I_{\geq}(x_{j},i) \neq \emptyset$ (resp. $I_{\leq}(x_{j},i) \neq \emptyset$)
equivalently means that $I_{\geq}(x_{j},i)$ (resp. $I_{\leq}(x_{j},i)$) is valid in the sense of Proposition~\ref{propineqValid}. 

\begin{defi}
  \label{defHierarchyxj}
  A variable $x_{j}$ is said to be

  \bit
\item $\min$-bounded if $I_{\geq}(x_{j}):=\cup_{i=1}^{m} I_{\geq}(x_{j},i) \neq \emptyset$. Otherwise
  $x_{j}$ is said to be $\min$-unbounded.

\item $\max$-bounded if $I_{\leq}(x_{j}):=\cup_{i=1}^{m} I_{\leq}(x_{j},i) \neq \emptyset$. Otherwise
  $x_{j}$ is said to be $\max$-unbounded.

  \item dominating if $x_{j}$ is $\min$-and-$\max$-bounded.
  \eit

  \end{defi}

From the previous definition we also define:

\begin{defi}[Bounded problem]
  A $(\max,+)$-linear programming problem (minimization or maximization) is
  bounded if the following condition holds.

  \begin{equation}
    \label{eqBoundedPB}
\boldsymbol{B}: \; \forall x_{j} \in \x, \; I_{\geq}(x_{j}) \cup I_{\leq}(x_{j}) \neq \emptyset.
    \end{equation}

  \end{defi}

The dominating variables have a variation domain smaller (in the sense of the set
inclusion) than the others variables. So that
there exists an order between variables denoted $\preceq_{var}$ based on their variation domain such that:

\begin{equation}
  \label{eqOrdreVar}
  \{x_{j} \in \x: \mbox{$I_{\geq}(x_{j})\neq \emptyset$ and $I_{\leq}(x_{j}) \neq \emptyset$}\} \preceq_{var}
  \{x_{j} \in \x: \mbox{$I_{\geq}(x_{j})\neq \emptyset$ and $I_{\leq}(x_{j}) =\emptyset$}\}.
\end{equation}

And of course we also mentioned that: \\
$\{x_{j} \in \x: \mbox{$I_{\geq}(x_{j})\neq \emptyset$ and $I_{\leq}(x_{j}) =\emptyset$}\}
  \preceq_{var}
  \{x_{j} \in \x: \mbox{$I_{\geq}(x_{j}) = \emptyset$ and $I_{\leq}(x_{j}) =\emptyset$}\}$.

  \begin{defi}
    \label{defhBhU}
  A non null $(x,h)$-linear function $f: (x,h) \mapsto \alpha^{\intercal} \otimes x \oplus \beta \otimes h$ is said to be
  $h$-bounded if $\beta \neq \dzero$. Otherwise the function $f$ is
  said to be $h$-unbounded.
  \end{defi}

We denote $h\B$ the set of $h$-bounded non null linear functions and we denote $h\U$ the set of
non null $h$-unbounded linear functions. And $(h\B, h\U)$ is a partition of the set of all non
null $(x,h)$-linear functions. We have the following interval calculus results.

\bit

\item $\forall f= \alpha^{\intercal} \otimes x \oplus \beta \otimes h \in h\B$, using interval calculus formulae we have:
\begin{subequations}
  \begin{equation}
    \label{eqfRdomaine}
  f(R_{\lambda \mu}(x,h)) =[\oplus_{i} \alpha_{i} \otimes \lambda \oplus \beta \otimes \lambda , +\infty],
  \end{equation}
  and by assumption ((\ref{eqmudominelambda}) with $\theta= \oplus_{i} \alpha_{i}$) on $\lambda$ and $\mu$ we have:
  
  \begin{equation}
    \label{eqDomainhB}
\forall \alpha \in \real_{\dzero}^{n}, \forall \beta \neq \dzero: \; f(R_{\lambda \mu}(x,h)) =[\beta \otimes \mu , +\infty]
    \end{equation}
  Let us stress that this interval does not depend on the $n$-dimensional row vector $\alpha$ and
  by assumption (\ref{eqmudominelambda}):
  \begin{equation}
    \label{eqDominclus}
    [\beta \otimes \mu , +\infty] \subset [\oplus_{i} \alpha_{i} \otimes \lambda, +\infty].
    \end{equation}
\end{subequations}

  \item  And $\forall f= \alpha^{\intercal} \otimes x \in h\U$:

    \begin{equation}
    f(R_{\lambda \mu}(x,h)) =[\oplus_{i} \alpha_{i} \otimes \lambda, +\infty].
    \end{equation}
    \eit

From this above results we have the following
global hierarchy between the non null linear functions denoted $\preceq_{fct}$ which is based on their
value domain:

\begin{equation}
  \label{eqOrdrefct}
\{f: f \in  h\B \} \preceq_{fct} \{f: f \in h\U \}.
  \end{equation}

\subsubsection{Choosing a variable for a new substitution after $k$ substitutions in \textsf{mpr}}
\label{subsubChoixvar}

In this section we describe the procedure for choosing a remaining variable to be
substituted. Assuming that $k$ substitutions have been done the cone associated with the
$\textsf{mpr}$ is denoted $\mathcal{C}(A^{+}, A^{-})^{[k]}$. The set of linear equalities at step
$k$, say $\mathcal{L}^{[k]}$, contains $k$ equalities of the form $\{x_{j}=f\}$ where $f$ is a $(x,h)$-linear
function. To lighten notations the superscript $[k]$ will not be applied on the vector of the
remaining variables $x$ and the set of the remaining variables $\x$ in the sequel. Recall that we use
numerotation convention~\ref{numMatcol} p. \pageref{numMatcol}.

In the next theorem we characterize a normal end of the substitution process.

\begin{theo}[Minimality and reachability at $\boldsymbol{-\infty}$ of $h \mapsto c_{h} \otimes h$]
  \label{thmchattaignable}
    The function $h \mapsto c_{h} \otimes h$ is the
  lower bound of the cost function $(x,h) \mapsto c^{\intercal}_{\x} \otimes x \oplus c_{h} \otimes h$ which is attained
  at $x=\boldsymbol{-\infty}$ iff the following condition holds:

  \begin{equation}
    \label{eqchok}
A_{[|1,m|] h}^{+[k]} \geq A_{[|1,m|] h}^{-[k]}.
    \end{equation}

  \end{theo}
\proof Considering the system of inequalities $A^{+[k]} \otimes w \geq A^{-[k]} \otimes w$. The
$[|1,m|] \times \x \cup \{h\}$ part
of the system is expressed as follows:

\[
A_{[|1,m|] \x}^{+[k]} \otimes x \oplus A_{[|1,m|] h}^{+[k]} \otimes h \geq A_{[|1,m|] \x}^{-[k]} \otimes x \oplus A_{[|1,m|] h}^{-[k]}
\otimes h.
\]
We have $z=c_{h} \otimes h$ at $x=\boldsymbol{-\infty}$ iff $(z,\boldsymbol{-\infty},h) \in \mathcal{C}(A^{+}, A^{-})^{[k]}$.
And based on the above inequality it means that $A_{[|1,m|] h}^{+[k]} \otimes h \geq A_{[|1,m|] h}^{-[k]} \otimes h$. And
this last inequality does not imply that $h=+\infty$ is the only solution iff $A_{[|1,m|] h}^{+[k]} \geq A_{[|1,m|] h}^{-[k]}$.
And the equivalence is proved. \cqfd \\

Before stating the next Theorem we need the following preliminary Lemma.

\begin{lemm}
      \label{lemmInegalitesgeq}
      Let us consider two tropical $(x,h)$-linear functions $f_{1}$ and $f_{2}$, a variable $y$ and the
      two inequalities $I_{\geq}(y,1):=(y \geq f_{1}(x,h))$ and $I_{\geq}(y,2):=(y \geq f_{2}(x,h))$. We assume that
      $(x,h)$ varies in some domain and that 
      the image of the domain by $f_{i}$ is the interval $\sigma_{i}:=[\lambda_{i},+\infty]$, $i =1,2$. 

      W. l. o. g. we can assume that $\lambda_{1} \leq \lambda_{2}$. Then, we have the following results.

\bit
\item (1). The maximum interval $\textsf{Max}_{\sigma \in \mathcal{I}}
  (\forall y \in \sigma \; \exists (x,h) \;
  (y \geq f_{1}(x,h) \mbox{ and } (y \geq f_{2}(x,h)) \mbox{ is true})$ is equal to
            $\textsf{Min}(\sigma_{1}, \sigma_{2})=[\lambda_{2},+\infty]$.

             \item (2). The maximum interval $\textsf{Max}_{\sigma \in \mathcal{I}}
            (\forall y \in \sigma \; \exists (x,h) \;
            (y \geq f_{1}(x,h) \mbox{ or } (y \geq f_{2}(x,h)) \mbox{ is true})$ is equal to
               $\textsf{Max}(\sigma_{1}, \sigma_{2})=[\lambda_{1},+\infty]$.
\eit
      
\end{lemm}
\proof \\

\noi
Proof of (1). We just have to remark that $\forall y \in \sigma \; (y \geq f_{1}(x,h) \mbox{ and }
(y \geq f_{2}(x,h)) \mbox{ is true})$ for $\sigma \in \mathcal{I}$ iff
$\sigma \subseteq [\lambda_{2},+\infty]= \textsf{Min}(\sigma_{1}, \sigma_{2})$. \\
Proof of (2). We just have to remark that $\forall y \in \sigma \;
(y \geq f_{1}(x,h) \mbox{ or } (y \geq f_{2}(x,h)) \mbox{ is true})$ for $\sigma \in \mathcal{I}$ iff
$\sigma \subseteq [\lambda_{1},+\infty]=\textsf{Max}(\sigma_{1}, \sigma_{2})$. \\

\cqfd \\

In the next Theorem we present the proof of the optimal choice for
one or several variables which can be substituted. Each variable is associated with an inequality
which belongs to a given set of inequalities. 

\begin{theo}
\label{theoNTZmin}
Let $\I$ be a subset of valid inequalities $I_{\geq}(x_{j},i)$, ie. a set
such that:
\begin{equation}
  \I^{\geq} \subseteq I_{\geq}(\x).
\end{equation}
Where
\begin{equation}
  \label{eqtteslesgeq}
I_{\geq}(\x):= \{(i,j): \mbox{ $x_{j} \in \x$ and $I_{\geq}(x_{j},i)$ is valid} \}.
  \end{equation}

We have the following implication:

\begin{equation}
  \label{implicationthmNTZ}
\mbox{ $I_{\geq}(x_{j},i)\neq \emptyset$ or equivalently $I_{\geq}(x_{j},i)$ is valid} \Rightarrow z \geq f_{z,ij}(x,h).
  \end{equation}
Recalling that $f_{z,ij}(x,h)$ is defined by (\ref{fzij2}). \\

Let $\T$ and $\T'$ denoting either $\B$ or $\U$, respectively. The $2$-tuple $(\T,\T')$ is defined
as a function of the set $\I$ according
to $\preceq_{fct}$ as follows (recalling the priority to the cost function).

First, let us define $\T$ as follows:

 \begin{equation}
    \label{eqdefTgeq}
    \T:=\left\{\bar{ll} \B & \mbox{ if $\exists (i,j) \in \I^{\geq}$ s.t. $f_{z,ij} \in h\B$} \\
    \U & \mbox{ otherwise}.
    \ear \right.
    \end{equation}
  
Then, let us define the following set:

\begin{equation}
  \label{mathcalTgeq}
  \mathcal{T}(\T,\I^{\geq}):=\{(i,j) \in \I^{\geq}: f_{z,ij} \in h\T \}.
  \end{equation}

We define the interval $\theta$ as:
   \begin{subequations}
     \begin{equation}
       \label{eqthetageq}
       \theta:= \textsf{Min}_{\{i: (i,j) \in \mathcal{T}(\T,\I^{\geq})\}} \theta_{i}
     \end{equation}
     with $\forall i: (i,j) \in \mathcal{T}(\T,\I^{\geq})$: 
   \begin{equation}
     \label{eqthetaIgeq}
\theta_{i}:= \textsf{Max}_{\{j: (i,j) \in \mathcal{T}(\T,\I^{\geq})\}} f_{z,ij}(R^{\T}_{\lambda \mu}(x,h)).
     \end{equation}

   \end{subequations}
   
   Then, we define the following logical decomposition of $\mathcal{T}(\T,\I^{\geq})$
   as follows:

     \begin{subequations}
\begin{equation}
  \label{eqdecTkIlogleq0}
  \mathcal{T}^{z}_{log}(\T, \I^{\geq}):= \textsf{AND}_{\{i:(i,j) \in \mathcal{T}(\T, \I^{\geq})\}} \mathcal{T}^{z}_{log}(\T, \I^{\geq})(i),
\end{equation}
with
\begin{equation}
    \label{eqdecTkIlogleq1}
    \mathcal{T}^{z}_{log}(\I, \I^{\geq})(i):= \textsf{OR}_{\{j:(i,j) \in \mathcal{T}(\I, \I^{\geq})\}} (z \geq f_{z,ij}).
      \end{equation}

     \end{subequations}

     We have the following results:

     \bit

   \item (1). The maximum interval
     $\textsf{Max}_{\sigma \in \mathcal{I}} (\sigma \mbox{ s.t. }\mbox{$\mathcal{T}^{z}_{log}(\T, \I^{\geq})$ is true})$
     is reached at $\sigma=\theta$.
\item (2). For all $i$ such that $(i,j) \in \mathcal{T}(\T, \I^{\geq})$ \\
  the maximum interval
  $\textsf{Max}_{\sigma \in \mathcal{I}} (\sigma \mbox{ s.t. } \mbox{$\mathcal{T}^{z}_{log}(\T, \I^{\geq})(i)$ is true})$
  is reached at $\sigma=\theta_{i}$.
     \eit
Where we use the following {\bf notation conventions}: \\
     \noi
     ``$\sigma \mbox{ s.t. } \mbox{$\mathcal{T}^{z}_{log}(\T, \I^{\geq})(i)$ is true}$'' means that: \\
     $\forall z \in \sigma \; \exists (x,h) \in R^{\T}_{\lambda \mu}(x,h): \;
     \textsf{OR}_{\{j:(i,j) \in \mathcal{T}(\I, \I^{\geq})\}} (z \geq f_{z,ij} (x,h))$ holds. \\
And
     \noi
     ``$(\sigma \mbox{ s.t. }\mbox{$\mathcal{T}^{z}_{log}(\T, \I^{\geq})$ is true})$'' means that:\\
     $\forall z \in \sigma \exists (x,h) \in R^{\T}_{\lambda \mu}(x,h): \\
     \textsf{AND}_{\{i:(i,j) \in \mathcal{T}(\T, \I^{\geq})\}}
     [ \textsf{OR}_{\{j:(i,j) \in \mathcal{T}(\I, \I^{\geq})\}} (z \geq f_{z,ij} (x,h)) ]$ holds. \\

   From the computation of $\theta$ the following cases may occur.

   \bit
 \item If $\textsf{argMin}(\theta) =\{(i^{*},j^{*})\}$ then
the set of substituable variables (which also contains the row
     where the variable appears in the constraints system) $S_{\theta}^{\geq}$ is defined by:
   \begin{equation}
          \label{Sthetageq}
S_{\theta}^{\geq}:= \textsf{argMin}(\theta), 
     \end{equation}
   which is of course a singleton.

   \item Otherwise $\textsf{n}(\textsf{argMin}(\theta)) \geq 2$ then let us define $\T'$ as follows:

   \begin{equation}
  \label{eqdefTT'geq}
  \T':= \left\{\bar{ll} \B & \mbox{ if $\exists (i,j) \in \textsf{argMin}(\theta)$ s.t. $f^{\geq}_{ij} \in h\B$} \\
  \U & \mbox{ otherwise.}
  \ear\right.
  \end{equation}
Then, let us define the following set:

\begin{equation}
  \label{mathcalTprimegeq}
  \mathcal{T}'(\T',\textsf{argMin}(\theta)):=\{(i,j) \in \textsf{argMin}(\theta): f^{\geq}_{ij} \in h\T' \}.
  \end{equation}

   Let us define interval $\theta'$ as:
   \begin{subequations}
     \begin{equation}
         \label{eqthetaprimegeq}
\theta':=\textsf{Min}_{\{i: (i,j) \in \mathcal{T}'(\T',\textsf{argMin}(\theta))\}} \theta'_{i},
     \end{equation}
     with
   \begin{equation}
     \label{eqthetaIprimegeq}
     \theta'_{i}:=\textsf{Max}_{\{j: (i,j) \in \mathcal{T}'(\T',\textsf{argMin}(\theta))\}} f^{\geq}_{ij}( R^{\T'}_{\lambda \mu}(x,h)).
     \end{equation}

\end{subequations}

Then, we define the following logical decomposition of $\mathcal{T}'(\T',\textsf{argMin}(\theta))$
   as follows:

     \begin{subequations}
\begin{equation}
  \label{eqdecTprimekIlogleq0}
  \mathcal{T}'_{log}(\T',\textsf{argMin}(\theta)):= \textsf{AND}_{\{i:(i,j) \in \mathcal{T}'(\T',\textsf{argMin}(\theta))\}}
  \mathcal{T}'_{log}(\T',\textsf{argMin}(\theta))(i),
\end{equation}
with
\begin{equation}
    \label{eqdecTprimekIlogleq1}
    \mathcal{T}'_{log}(\T',\textsf{argMin}(\theta))(i):= \textsf{OR}_{\{j:(i,j) \in \mathcal{T}'(\T',\textsf{argMin}(\theta))\}}
    (x_{j} \geq f^{\geq}_{ij})).
      \end{equation}

     \end{subequations}

          We have the following results:

     \bit

   \item (1'). The maximum interval
     $\textsf{Max}_{\sigma \in \mathcal{I}} (\sigma \mbox{ s.t. }\mbox{$\mathcal{T}'_{log}(\T',\textsf{argMin}(\theta))$ is true})$
     is reached for $\sigma=\theta'$.
\item (2'). For all $i$ such that $(i,j) \in  \mathcal{T}'(\T',\textsf{argMin}(\theta))$ the maximum interval
   $\textsf{Max}_{\sigma \in \mathcal{I}} (\sigma \mbox{ s.t. } \mbox{$\mathcal{T}'_{log}(\T',\textsf{argMin}(\theta))(i)$ is true})$ is reached
  for $\sigma=\theta'_{i}$.
     \eit
Where we use the following {\bf notation conventions}: \\
     \noi
``$\sigma \mbox{ s.t. } \mbox{$\mathcal{T}'_{log}(\T',\textsf{argMin}(\theta))(i)$ is true}) $'' means that: \\
     $\forall z \in \sigma \; \exists (x,h) \in R^{\T}_{\lambda \mu}(x,h): \;
     \textsf{OR}_{\{j:(i,j) \in \mathcal{T}'(\T',\textsf{argMin}(\theta))\}} (x_{j} \geq f_{ij} (x,h))$ holds. \\
And
     \noi
     ``$(\sigma \mbox{ s.t. } \mbox{$\mathcal{T}'_{log}(\T',\textsf{argMin}(\theta))(i)$ is true})$'' means that: \\
     $\forall z \in \sigma \; \exists (x,h) \in R^{\T}_{\lambda \mu}(x,h)$: \\
     $\textsf{AND}_{\{i: (i,j) \in \mathcal{T}'(\T',\textsf{argMin}(\theta))\}}
     [   \textsf{OR}_{\{j:(i,j) \in \mathcal{T}'(\T',\textsf{argMin}(\theta))\}}(x_{j} \geq f_{ij} (x,h)) ]$ holds. \\

     If $S_{\theta'}^{\geq}$ denotes the set of subsituable variables (which also contains the row
     where the variable appears in the constraints system) we have:

   \begin{equation}
     \label{Sthetaprimeleq}
S_{\theta'}^{\geq}:= \textsf{argMin}(\theta').
     \end{equation}

   \eit

\end{theo}

\proof First, let us prove results (1) and (2) of the Theorem for the cost variable $z$. The result ((2), this Theorem) for the intervals
$\theta_{i}$ (see (\ref{eqthetaIgeq})) is just a generalization of result
((2), Lemma~\ref{lemmInegalitesgeq}) based on the
associativity of $\textsf{Min}$ of result ((2), Lemma~\ref{lemmInegalitesgeq}) to the case of
$\textsf{n}(\{j: (i,j) \in \mathcal{T}(\T,\I^{\geq})\})$ disjunctions of the form $(z \geq f_{z,ij})$.
Each interval $f_{z,ij}(R^{\T}_{\lambda \mu}(x,h))$ is an element of $\mathcal{I}$ thus $\theta_{i}$ is
the union of several intervals $f_{z,ij}(R^{\T}_{\lambda \mu}(x,h))$ and is also an element of $\mathcal{I}$
(see again the particular case of two disjunctions treated in ((2), Lemma~\ref{lemmInegalitesgeq})).
Thus, the following assertion:

\[
\forall z \in \theta_{i} \; \exists (x,h) \in R^{\T}_{\lambda \mu}(x,h): \;
     \textsf{OR}_{\{j:(i,j) \in \mathcal{T}(\I, \I^{\geq})\}} (z \geq f_{z,ij} (x,h)) 
\]
is true for $\theta_{i}$ defined by (\ref{eqthetaIgeq}) which is maximal in the sense of the interval inclusion. \\

The conjunction $\textsf{AND}_{\{i:(i,j) \in \mathcal{T}(\T, \I^{\geq})\}} [\cdot]$ is true iff
$\forall i \in \{i:(i,j) \in \mathcal{T}(\T, \I^{\geq})\}$, $z \in \theta_{i}$. In other words,
the conjunction is true iff $z$ is an element of the interval $\cap_{\{i:(i,j) \in \mathcal{T}(\T, \I^{\geq})\}} \theta_{i}$.
And because the $\theta_{i}$'s are element of $\mathcal{I}$ we have
(see equality (\ref{eq-intersection=Min-interval})):

\[
\bar{lll}
\cap_{\{i:(i,j) \in \mathcal{T}(\T, \I^{\geq})\}} \theta_{i} & = & \textsf{Min}_{\{i:(i,j) \in \mathcal{T}(\T, \I^{\geq})\}} \theta_{i} \\
\mbox{} & = & \theta.
\ear
\]
And $\theta$ is maximal. \\
The results (1) and (2) of this Theorem for the cost variable $z$ are now proved. \\

\noi
Second, let us prove the results (1') and (2') of this Theorem for $\theta'$ when $\textsf{n}(\textsf{argMin}(\theta)) \geq 2$. \\
In this case $\forall i \in \{i: (i,j) \in \mathcal{T}'(\T',\textsf{argMin}(\theta))\}$
all the variables $x_{j}$, $j \in \{j: (i,j) \in \mathcal{T}'(\T',\textsf{argMin}(\theta))\}$
can be exchanged. It means that any inequality $(x_{j} \geq f^{\geq}_{ij})$ can be replaced by
a generic inequality $(y \geq f^{\geq}_{ij})$ where $y$ is a generic variable which stands
for any $x_{j}$, $j \in \{j: (i,j) \in \mathcal{T}'(\T',\textsf{argMin}(\theta))\}$. Therefore,
the proof of the expression of the intervals $\theta'_{i}$ (see (\ref{eqthetaIprimegeq})) use
the same arguments as the ones use for the $\theta_{i}$'s. And the proof of the expression of
the interval $\theta'$ (see (\ref{eqthetaprimegeq})) use also the same arguments. The proofs only differ
by their reference set where the refernce set $\mathcal{T}(\T,\I^{\geq})$ (for $z$) is replaced by the
the refernce set $\mathcal{T}'(\T',\textsf{argMin}(\theta))$ for the $x_{j}$'s. The proof is
not developed here.

\cqfd. \\

\subsubsection{The substitution procedure at step $k$ itself}
\label{subsubSubstit-cases}

Assuming that the substitution process is not finished we are now in position to indicate
how to choose a remaining variable to be substituted. The choice also depends on the
following different cases which are listed hereafter.

Let us define the following noticeable sets.

\begin{subequations}
  The set of dominating variables:
\begin{equation}
  \label{eqDomvar}
D:=\{x_{j} \in \x: \; I_{\leq}(x_{j}) \neq \emptyset \mbox{ and }
I_{\geq}(x_{j}) \neq \emptyset\}.
\end{equation}

Then the set of the valid inequalities is defined as follows: 

\begin{equation}
  \label{eqdefID}
  \I^{\geq}:=\left\{\bar{ll} I_{\geq}(\x \cap D) & \mbox{ if $D \neq \emptyset$} \\
  I_{\geq}(\x), \mbox{ cf. (\ref{eqtteslesgeq})}           & \mbox{ otherwise.}
  \ear \right.
  \end{equation}
Where: $I_{\geq}(\x \cap D):= \{(i,j): \mbox{ $x_{j} \in \x \cap D$ and $I_{\geq}(x_{j},i)$ is valid} \}$.
We adopt the convention that $I_{\geq}(\emptyset)=\emptyset$.

\end{subequations}

\noi
\textsc{Case $0$}: $min(z)=+\infty$ or no minimum. \textsc{Case $0.1$}: $\x = \emptyset$ and the condition (\ref{eqchok}) of
Theorem~\ref{thmchattaignable} does not holds. $\x=\emptyset$ means that
$k=n$ (no more variable to be substituted) and $\textsf{cost}^{[n]}(x,h)= c_{h}^{[n]} \otimes h$. The condition (\ref{eqchok}) of
Theorem~\ref{thmchattaignable} does not holds means $A_{[|1,m|] h}^{+[n]} \ngeq
A_{[|1,m|] h}^{-[n]}$ then $h=+\infty$ is the only possible value for $h$ and by backward substitution we have
$z=+\infty$, $x=\boldsymbol{+\infty}$. \\

\noi 
\textsc{Case $0.2$}:
$\x \neq \emptyset$ and $I_{\leq}(\x)=\emptyset$ and $I_{\geq}(\x)=\emptyset$ and the condition (\ref{eqchok}) of
Theorem~\ref{thmchattaignable} does not holds. The substitution process stops because no
valid inequalities can be generated. No minimum reached. \\

\noi
\textsc{Case $1$}: {\bf switching case}, cf. \S~\ref{subswitch}. $\x^{+}=\emptyset$ and $\textsf{cost}^{[k]}(x,h)= c_{h}^{[k]} \otimes h$ and the
condition (\ref{eqchok}) of Theorem~\ref{thmchattaignable} does not holds and $\x^{\dzero} \neq \emptyset$ and
$I_{\geq}(\x^{\dzero}) = \emptyset$ and $I_{\leq}(\x^{\dzero}) \neq \emptyset$. The substitution must switch to
the following maximizing problem:
\begin{subequations}
  \begin{equation}
\textsf{PMR}(A,b,C,d,c,c_{h},z,x,h,(m,n)),
    \end{equation}
  where the matrices $A$ and $C$ are defined by:
 \begin{equation}
    A:=A^{-[k]}_{[|1,m|][|1,n|]}, \; C:=A^{+[k]}_{[|1,m|][|1,n|]}. 
 \end{equation}
The vectors $b$ and $d$ are defined by:
 \begin{equation}
   b:= A^{-[k]}_{[|1,m|]h}, d:= A^{+[k]}_{[|1,m|]h}.
       \end{equation}
 The initial cost function of the \textsf{MPR} is characterized by the vector
 $c$ and the constant $c_{h}$ such that:
 \begin{equation}
    c:=\boldsymbol{-\infty}, \; c_{h}:=c_{h}^{[k]}. 
    \end{equation}
 And the variables of the problem are:
 \begin{equation}
    z:=z , \; x \leftarrow x^{\dzero}, \; h:=h.
 \end{equation}
 Where the notation $x \leftarrow x^{\dzero}$ means that we store in the $n$-dimensional vector
 $x$ the remaining variables of the set $\x^{\dzero}$. The other components (already substituted
 variables) are set to $\dzero$. \\
 
 And the dimensions $(m,n)$ of the \textsf{MPR} are obviously the same
 as in the \textsf{mpr}.
  \end{subequations}

\noi
\textsc{Case $2$}: Not \textsc{Case $0$} and not \textsc{Case $1$}. The first case is \textsc{Case $2.1$}: $\x = \emptyset$ and the condition (\ref{eqchok}) of
Theorem~\ref{thmchattaignable} holds.  This means that
$k=n$ (no more variable to be substituted) and $\textsf{cost}^{[n]}(x,h)= c_{h}^{[n]} \otimes h$ and $A_{[|1,m|] h}^{+[n]} \geq
A_{[|1,m|] h}^{-[n]}$ then $h$ can take all possible values in particular we can take $h=\dun$ and the $\textsf{mpr}$ succeeds
with finite minimum $z=c_{h}^{[n]}$. \\

The other case is \textsc{Case $2.2$}: $\x \neq \emptyset$ and $\exists x_{j} \in \x$ such that: $I_{\geq}(x_{j}) \neq \emptyset$.
Let us distinguish the following subcases which are based on the form
of the cost function:

\begin{subequations}
\begin{equation}
  \textsf{cost}^{[k]}(x,h)= c^{\intercal}_{\x^{+}} \otimes x^{+} \oplus c^{\intercal}_{\x^{\dzero}} \otimes x^{\dzero} \oplus c_{h} \otimes h,
  \end{equation}
where:
\begin{equation}
  \label{eqpartitiondex}
  \mbox{$\x^{+}:=\{x_{j} \in \x: \; c_{j} > \dzero\}$ and $\x^{\dzero}:=\{x_{j} \in \x: \; c_{j} = \dzero\}$}.
  \end{equation}
And the vectors $x^{+}$ and $x^{\dzero}$ are defined by
\begin{equation}
  \label{eqpartitiondex2}
x^{+}:=(x_{j})_{\{j: x_{j} \in \x^{+}\}} \mbox{ and } x^{\dzero}:=(x_{j})_{\{j: x_{j} \in \x^{\dzero}\}}.
 \end{equation}
\end{subequations}

\noi
\textsc{Case $2.2.1$}: the condition (\ref{eqchok}) of Theorem~\ref{thmchattaignable} holds. Then, $z=c_{h} \otimes h$ is
the minimum for the $\textsf{mpr}$. And set $x$ to $\boldsymbol{-\infty}$. \\

\noi
\textsc{Case $2.2.2$}: otherwise, the condition (\ref{eqchok}) of Theorem~\ref{thmchattaignable} does not hold. \\

\noi
We apply Theorem~\ref{theoNTZmin} with the set of valid inequalities $\I^{\geq}$ defined by (\ref{eqdefID}). \\

\noi
We get a $2$-tuple $(i^{*},j^{*})$ for the substitution
of $x_{j^{*}}=f^{\geq}_{i^{*}j^{*}}(x,h)$. \\

\noi
We apply results of \S~\ref{subsecNewcost+Newconepmrp} to compute the new set
of equalities $\mathcal{L}^{[k+1]}$, the new cost function $\text{cost}^{[k+1]}(x,h)$ and the new cone
 $\mathcal{C}(A^{+},A^{-})^{[k+1]}$.

\subsubsection{The new characteristic elements of the \textsf{mpr} after substitution computed in
$\mathcal{O}(mn^{2})$}
  \label{subsecNewcost+Newconepmrp}

  The new set of equalities is defined as follows. If conditions of Theorem~\ref{thmchattaignable} are verified then
  set $k+1=n$ and $\mathcal{L}^{[n]}:= \mathcal{L}^{[k]} \cup_{x_{j} \in \x}\{x_{j}= \dzero\} \cup \{z=c_{h} \otimes h\}$. Then,
  $\textsf{pmrp}$
    {\bf stops} and {\bf we solve} the
    linear system $\mathcal{L}^{[n]}$ and express all the $x_{j}$'s as function of the homogenization variable $h$. Otherwise,
    we define $\mathcal{L}^{[k+1]}$ as: $\mathcal{L}^{[k+1]}:= \mathcal{L}^{[k]} \cup \{x_{j^{*}}=f_{i^{*}j^{*}}(x,h)\}$ and we
    apply Theorem~\ref{theoMinCi} of \S~\ref{coneMinCi}. And based on the previous cases the new cost
  function is defined as:

  \begin{equation}
  \textsf{cost}^{[k+1]}(x,h):= f_{z,i^{*}j^{*}}(x,h).
  \end{equation}
  
And in all cases the new
cone $\mathcal{C}(A^{+}, A^{-})^{[k+1]}$ associated with $\textsf{pmrp}$ is deduced from the
following set of inequalities:

  \begin{subequations}
    \begin{equation}
      \bar{l}
      z \geq \textsf{cost}^{[k+1]}(x,h), \\
      \forall i \in [|1,m|]: \; a^{+[k]}_{i,.} \otimes \ell \geq  a^{-[k]}_{i,.} \otimes \ell .
      \ear
 \end{equation}
    and $\ell$ is the $n+2$-dimensional column vector which has the same components
    as the vector $w^{[k]}$ except its $j^{*}$th component wich is:
    \begin{equation}
\ell_{j^{*}}:= f_{i^{*}j^{*}}(x,h).
      \end{equation}
  \end{subequations}

  Let us define the following $(n+2) \times (n+2)$ {\em transition matrix} $T^{k \rightarrow k+1}$ by:

  \begin{equation}
    \forall j: t^{k \rightarrow k+1}_{j,.}:= \left\{ \bar{lc} e^{\intercal}_{j}  \mbox{ if $j \neq j^{*}$} \\
                                                (\dzero, v^{\intercal}_{i^{*}j^{*}}, r_{i^{*}j^{*}}) & \mbox{ if $j=j^{*}$.}
    \ear \right.
  \end{equation}
  Where $e^{\intercal}_{j}$ denotes the $j$th $n+2$-dimensional row vector of the
  $(n+2) \times (n+2)$-identity matrix $\bI_{n+2}$.

  The new matrices $(A^{+}, A^{-})^{[k+1]}$ are defined by:
  \begin{equation}
(A^{+}, A^{-})^{[k+1]}:= \textsf{setrowtozero}(A^{+[k]} \otimes T^{k \rightarrow k+1}, A^{-[k]} \otimes T^{k \rightarrow k+1})
    \end{equation}
recalling that function \textsf{setrowtozero} is defined by (\ref{ligneazero}).

And the new vector $w^{[k+1]}$ is defined by:
\begin{equation}
w^{[k+1]}:= \left(\bar{c}
     z \\
     \vdots \\
     \dzero \\
     \vdots \\
     h
     \ear \right) \; \bar{c}
     \mbox{} \\
     \mbox{} \\
     \leftarrow \; j^{*} \\
     \mbox{} \\
     \mbox{}
     \ear 
  \end{equation}
And the new set $\x$ of the remaining variables is defined by:
\begin{equation}
\x := \x \setminus \{x_{j^{*}}\}.
  \end{equation}

\section{Linear fractional programming}
\label{seclinfraction}

In this section we consider the following $(\max,+)$-analogue of the
linear fractional programming (see \cite{kn:ChCoop1962}) defined as follows.

\begin{subequations}
\begin{equation}
  \label{eqdefminfracprog}
\min\{z= p^{\intercal} \otimes x / r^{\intercal} \otimes x, x \in \mathcal{P}(A,b,C,d) \cup \{\boldsymbol{\pm \infty}\} \}.
  \end{equation}
Where $p$, $r$ are two $n$-dimensional vectors satisfying the following
assumption:

\begin{equation}
\label{eqAssumpqrs}
\mbox{{\bf Assumption $\boldsymbol{1^{f}}$}}:\; \mbox{$p \neq \boldsymbol{-\infty}$ and $r \neq \boldsymbol{-\infty}$}.
  \end{equation}
  
\end{subequations}

This kind of problem has been studied in eg. \cite{kn:GKS012} and \cite{kn:GMH014} where pseudo-polynomial
algorithms have been provided.

Hereafter, we follow the $(\max,+)$-analogue of the Charnes-Cooper transformation \cite{kn:ChCoop1962}
which has been used in \cite{kn:GMH014}. Using also the Fourier's trick the problem (\ref{eqdefminfracprog}) is unfolded
as follows:

\begin{equation}
  \bar{l}
  \min(z) \\
z \geq p^{\intercal} \otimes y \; : \mbox{Fourier's trick} \\
A \otimes y \oplus b \otimes t \geq C \otimes y \oplus d \otimes t \; :\mbox{$(y,t)$-cone part}\\
r^{\intercal} \otimes y \geq \dun \; : \mbox{$(y)$-polyhedron part} \\
t \geq \dzero  \; : \mbox{$t$ added variable such that $y = t \otimes x$}.
\ear
\end{equation}

\begin{rem}
  In \cite{kn:GMH014} the authors follow exactly the Charnes-Cooper transformation. This means that
  the $(y)$-polyhedron part is not an inequality but the equality: $r^{\intercal} \otimes y = \dun$.
  This latter equality is equivalent to: (1). $r^{\intercal} \otimes y \geq \dun$ and
  (2). $\dun \geq r^{\intercal} \otimes y $. But substitution allows us to drop the inequality (2) because when an
  inequality is chosen using our criteria developed in \S~\ref{subsubChoixvar} then we use the result
  of Theorem~\ref{theoMinCi} of \S~\ref{coneMinCi} to saturate this inequality, ie. $\geq$ becomes $=$.
  \end{rem}

Now, we just have to homogenize the $(y)$-polyhedron part by adding the homogeneous variable $h$ and we obtain the following
\textsf{mpr} problem: 

\begin{subequations}
  \begin{equation}
      \label{eqprobFractàresoudre}
  \bar{l}
  \min(z) \\
  z \geq p^{\intercal} \otimes y  \; : \mbox{Fourier's trick}.
  \ear
  \end{equation}
The research of a minimum is developed on the following cone:
\begin{equation}
  \label{eqconefractprog}
    \bar{l}
A \otimes y \oplus b \otimes t \geq C \otimes y \oplus d \otimes t \; :\mbox{$(y,t)$-cone part}\\
r^{\intercal} \otimes y  \geq h \; : \mbox{$(y)$-polyhedron part homogenized}.
\ear
  \end{equation}
The variable $t$ used in the Charnes-Cooper transformation is such that:
\begin{equation}
  \label{eqVartety}
t \geq \dzero, \; y = t \otimes x.
\end{equation}
And the homogenization variable $h$ satifies:
  \begin{equation}
      \label{eqprobFractavec-h}
h > \dzero.
  \end{equation}
And
\begin{equation}
  \label{zhnonsubstifract}
\mbox{$z$ and $h$ are not substituable}.
\end{equation}
Finally, among the variables of the problem $z, y_{1}, \ldots, y_{n},t$ we have:
\begin{equation}
  \label{zpriofract}
  \mbox{$z$ has the highest priority because bounded by the cost function}.
  \end{equation}

\end{subequations}

\section{Numerical examples}
\label{secExNum}
In ths section we study pedagogical examples. In \S~\ref{subcoherencyexample} we illustrate the coherency of our strongly polynomial
method on a toy example. 

In \S~\ref{subExminimization-de-Gaubertetal}  (problem of minimization) and in \S~\ref{subExfractional-de-Gaubertetal} (problem
of fractional minimization) we illustrate our strongly polynomial method on some numerical
examples borrowed from the litterature which are solved using pseudo-polynomial method in \cite{kn:GKS012}.

\subsection{A minimum which is equal to $+\infty$}
\label{subcoherencyexample}
In this subsection we consider the following minimization problem:
the homogenization variable $h$:
\[
\min\{z=x_{1}, x=(x_{1}) \in \mathcal{P}(A,b,C,d) \cup \{\boldsymbol{\pm \infty}\}\},
\]
where the polyhedron $\mathcal{P}(A,b,C,d)$ is characterized by the following vectors and matrices:
\[
A:=\left(\bar{c} \dun \\ 2^{\otimes (-1)} \ear \right), \; b:=\left(\bar{c} \dun \\ 1 \ear \right), \;
C:=\left(\bar{c} 1^{\otimes (-1)} \\ \dun \ear \right), \; d:=\left(\bar{c} 2 \\ \dzero \ear \right).
\]
Because the example is very small we do not need to describe the cone associated with \textsf{mpr}. Using
the homogenization variable $h$ the \textsf{mpr} is unfolded as follows:

\begin{equation}
  \label{pmrCoherencyEx0}
  \bar{l}
  \min(z) \\
 l_{0}: \; z \geq x_{1} \\
 l_{1}: \;  x_{1} \oplus h \geq 1^{\otimes (-1)} \otimes x_{1} \oplus 2 \otimes h  \\
 l_{2}: \; 2^{\otimes (-1)} \otimes x_{1} \oplus 1 \otimes h \geq x_{1}.
  \ear
\end{equation}
Only $l_{1}$ generates the valid inequality $I_{\geq}(x_{1},1)=\{x_{1} \geq 2 \otimes h\}$. Thus,
the substitution of $x_{1}$ by $2 \otimes h$ (ie. we have
$\mathcal{L}^{[1]}= \{x_{1} = 2 \otimes h\}$) provides the new system:
\begin{equation}
  \label{pmrCoherencyEx1}
  \bar{l}
  \min(z) \\
 l_{0}: \; z \geq 2 \otimes h \\
 l_{1}: \;  2 \otimes h \geq  2 \otimes h  \\
 l_{2}: \; 1 \otimes h \geq 2 \otimes h 
  \ear
\end{equation}
Constraint $l_{2}$ implies that $h=+\infty$ is the only possible value for $h$, where $h$ is assumed
to be $> \dzero$. Then, by backward substitution
we have: $z=2 \otimes +\infty=+\infty$ and
$x_{1}= 2 \otimes +\infty=+\infty$. And $+\infty$ is element of the set $\mathcal{P}(A,b,C,d) \cup \{\boldsymbol{\pm \infty}\}$. 

This illustrates the coherency of the method because $\min(z)=+\infty$ is actually reached at $x=(+\infty)$ which is
a point of the polyhedron $\cup \{\boldsymbol{\pm \infty}\}$. 

\subsection{Minimization Example}
\label{subExminimization-de-Gaubertetal}
In this subsection we consider the Example 2 p. 1463 of \cite{kn:GKS012} which is:

\[
\min\{z= 2 \otimes x_{1} \oplus 4^{\otimes (-1)} \otimes x_{2}, x=(x_{1},x_{2})^{\intercal} \in \mathcal{P}(A,b,C,d) \cup \{\boldsymbol{\pm \infty}\}\},
\]
Where:
\[
A:=\left(\bar{cc}
2^{\otimes(-1)} & \dun \\
  \dun & 1^{\otimes (-1)} \\
  1 & 2^{\otimes (-1)} \\
  2 & \dzero \\
  \dun & \dzero \\
  2^{\otimes (-1)} & \dzero \\
  4^{\otimes (-1)} & \dzero
  \ear \right), \;b:=\left(\bar{c}
  \dzero \\
  \dzero \\
  \dzero \\
  \dzero \\
  \dun \\
  \dun \\
  \dun   
  \ear \right), \;
  C:=\left(\bar{cc}
  \dzero & \dzero \\
  \dzero & \dzero \\
  \dzero & \dzero \\
  \dzero & 3^{\otimes (-1)} \\
  \dzero & 4^{\otimes (-1)} \\
  \dzero & 5^{\otimes (-1)} \\
  \dzero & 6^{\otimes (-1)}
  \ear \right), \; d:=\left(\bar{c}
  \dun \\
  \dun \\
  \dun \\
  \dun \\
  \dzero \\
  \dzero \\
    \dzero
  \ear \right).
\]
with the following change of notations $A \leftrightarrow B$, $b \leftrightarrow d$, $C \leftrightarrow A$ and
$d \leftrightarrow c$.

The $\textsf{mpr}$ is defined on the cone $\mathcal{C}(A,b,C,d)$ obtained from the
polyhedron $\mathcal{P}(A,b,C,d)$ by adding the homogenization variable $h$ as already explained
previuously.

At step $0$ of the substitution the cone associated with the above constraints of the $\textsf{mpr}$ problem
is denoted $\mathcal{C}(A^{\; +}, A^{ \; -})^{[0]}$ and is defined by the following system of inequalities:
\begin{subequations}
\begin{equation}
  \label{thelinsysofineq0}
A^{+[0]} \otimes w^{[0]} \geq  A^{-[0]} \otimes w^{[0]},
\end{equation}
with $w^{[0]}= (z , x_{1} , x_{2}  , h)^{\intercal} \in \overline{\real}_{\dzero}^{4}$ and the matrices $(A^{\; +}, A^{ \; -})^{[0]}$
are defined by:

\begin{equation}
A^{+[0]}:= \left(\bar{cccc}
\dun & \dzero & \dzero & \dzero \\
\dzero & 2^{\otimes(-1)} & \dun & \dzero \\
\dzero &  \dun & 1^{\otimes (-1)} & \dzero \\
\dzero &   1 & 2^{\otimes (-1)} & \dzero \\
\dzero &  2 & \dzero & \dzero \\
\dzero & \dun & \dzero & \dun \\
\dzero &  2^{\otimes (-1)} & \dzero & \dun  \\
\dzero &  4^{\otimes (-1)} & \dzero  & \dun
\ear \right), \;
A^{-[0]}:= \left(\bar{cccc}
\dzero & 2 & 4^{\otimes (-1)} & \dzero \\
\dzero & \dzero & \dzero &  \dun \\
\dzero & \dzero & \dzero  &  \dun \\
\dzero &  \dzero & \dzero  &  \dun \\
\dzero & \dzero & 3^{\otimes (-1)}  &  \dun \\
\dzero &   \dzero & 4^{\otimes (-1)} & \dzero \\
\dzero &  \dzero & 5^{\otimes (-1)} & \dzero \\
\dzero &  \dzero & 6^{\otimes (-1)} & \dzero
\ear \right).
\end{equation}
\end{subequations}

And the Fourier's trick for this example is:
\begin{subequations}
\begin{equation}
\min(z)
  \end{equation}

\begin{equation}
z \geq \textsf{cost}^{[0]}(x,h)
  \end{equation}

\begin{equation}
\textsf{cost}^{[0]}(x,h)= 2 \otimes x_{1} \oplus 4^{\otimes (-1)} \otimes x_{2} \oplus \dzero \otimes h.
  \end{equation}

\end{subequations}

$\x=\{x_{1}, x_{2}\} \neq \emptyset$ and the $\textsf{mpr}$ problem is bounded because one can check
that: $\forall x_{j} \in \x$: $I_{\geq}(x_{j}) \cup I_{\leq}(x_{j}) \neq \emptyset$. We have
$\x^{+}=\{x_{1}, x_{2}\}$ and $\x^{\dzero}=\emptyset$. The set of dominating variables is
$D=\{x_{2}\}$ (see (\ref{eqDomvar})). It is easy to see here that $I_{\leq}(x_{1})=\emptyset$ because
$A_{[|1,m|] x_{1}}^{-[0]} = \boldsymbol{-\infty}$. Clearly, the condition (\ref{eqchok}) of Theorem~\ref{thmchattaignable} does
not hold. \\

\noi
\textsc{Case $2.2.2$}, \S~\ref{subsubSubstit-cases} applies. We have (see \ref{eqdefID}):
$\I^{\geq}=I_{\geq}(x_{2})=\{(1,2), (2,2), (3,2)\}$. \\
We apply Theorem~\ref{theoNTZmin} with $\I^{\geq}=\{(1,2), (2,2), (3,2)\}$. \\

From the following array of possible inequalities provided using (\ref{a+a-ineq})-(\ref{ORa+a-ineq}):

\[
\bar{lllllll}
 \left( \bar{c} z \\ x_{2}
\ear \right)_{1} & \geq & \left( \bar{c}  2 \otimes x_{1} \oplus 4^{\otimes (-1)} \otimes h\\
h \ear \right) & :: & \left(\bar{c}
\mbox{$[4^{\otimes (-1)} \otimes \mu, +\infty]$} \\ \mbox{$[ \dun \otimes \mu,+\infty ]$} \ear \right) & :: & \left(\bar{c}
h\B \\  h\B \ear \right) \\
 \left( \bar{c} z \\ x_{2}
\ear \right)_{2} & \geq & \left( \bar{c}  2 \otimes x_{1} \oplus 3^{\otimes (-1)} \otimes h\\
1 \otimes h \ear \right) & :: & \left(\bar{c}
\mbox{$[3^{\otimes (-1)} \otimes \mu, +\infty]$} \\ \mbox{$[ 1 \otimes \mu,+\infty ]$} \ear \right) & :: & \left(\bar{c}
h\B \\  h\B \ear \right) \\
 \left( \bar{c} z \\ x_{2}
\ear \right)_{3} & \geq & \left( \bar{c}  2 \otimes x_{1} \oplus 2^{\otimes (-1)} \otimes h\\
2 \otimes h \ear \right) & :: & \left(\bar{c}
\mbox{$[2^{\otimes (-1)} \otimes \mu, +\infty]$} \\ \mbox{$[ 2 \otimes \mu,+\infty ]$} \ear \right) & :: & \left(\bar{c}
h\B \\  h\B \ear \right).
\ear
\]

one has $\T=\B$ (see (\ref{eqdefTgeq})), $ \mathcal{T}(\B,\I^{\geq})=\{(1,2), (2,2), (3,2)\}$ (see (\ref{mathcalTgeq})),
$\textsf{argMin}(\theta)=\{(3,2)\}$ with $\theta=[2^{\otimes (-1)} \otimes \mu, +\infty]$ defined by (\ref{eqthetageq}),
$S_{\theta}^{\geq}=\{x_{2}\}$ (see (\ref{Sthetageq})). $\textsf{argMin}(\theta)$ is a singleton thus: \\
$(i^{*},j^{*})=(3,2)$ and $x_{2}=2 \otimes h$. \\

We apply the update procedure of \S~\ref{subsecNewcost+Newconepmrp} and the new elements of the problem are as
follows. \\

The set of linear equalities is now:

\begin{equation}
\mathcal{L}^{[1]}= \{x_{2} = 2 \otimes h\}.
\end{equation}

The new cost function is:
\begin{equation}
  \textsf{cost}^{[1]}(x,h):= f_{z,32}(x,h)=2 \otimes x_{1} \oplus 2^{\otimes (-1)} \otimes h.
\end{equation}

The $4 \times 4$-transition matrix $T^{0 \rightarrow 1}$ is defined by:

\begin{equation}
  T^{0 \rightarrow 1}:= \left(\bar{cccc}
  \dun & \dzero & \dzero & \dzero \\
    \dzero & \dun & \dzero & \dzero \\  
  \dzero & \dzero & \dzero & 2 \\
  \dzero & \dzero & \dzero & \dun 
  \ear \right). 
  \end{equation}

We calculate the following matrices $A^{+}:=A^{+[0]} \otimes T^{0 \rightarrow 1}$ and $A^{-}:=A^{-[0]} \otimes T^{0 \rightarrow 1}$ and we obtain:

\[
A^{+}:= \left(\bar{cccc}
\dun & \dzero & \dzero & \dzero \\
\dzero & 2^{\otimes(-1)} & \dzero & 2 \\
\dzero &  \dun & \dzero & 1 \\
\dzero &   1 & \dzero & \dun \\
\dzero &  2 & \dzero & \dzero \\
\dzero & \dun & \dzero & \dun \\
\dzero &  2^{\otimes (-1)} & \dzero & \dun  \\
\dzero &  4^{\otimes (-1)} & \dzero  & \dun
\ear \right), \;
A^{-}:= \left(\bar{cccc}
\dzero & 2 & \dzero & 2^{\otimes (-1)} \\
\dzero & \dzero & \dzero &  \dun \\
\dzero & \dzero & \dzero  &  \dun \\
\dzero &  \dzero & \dzero  &  \dun \\
\dzero & \dzero & \dzero  &  \dun \\
\dzero &   \dzero & \dzero & 2^{\otimes (-1)} \\
\dzero &  \dzero & \dzero & 3^{\otimes (-1)} \\
\dzero &  \dzero & \dzero & 4^{\otimes (-1)}
\ear \right).
\]

We remark that $\forall i \neq 0, 4$ we have $a^{+}_{i,.} \geq a^{-}_{i,.}$ thus the
new cone $\mathcal{C}(A^{\; +}, A^{ \; -})^{[1]}$ is defined as follows:

\begin{subequations}
\begin{equation}
  \label{thelinsysofineq1}
A^{+[1]} \otimes w^{[1]} \geq  A^{-[1]} \otimes w^{[1]},
\end{equation}
with $w^{[1]}= (z, x_{1}, \dzero , h)^{\intercal} \in \overline{\real}_{\dzero}^{4}$ and the matrices $(A^{\; +}, A^{ \; -})^{[1]}$
are defined as the result of $\textsf{setrowtozero}(A^{+}, A^{-})$:

\begin{equation}
  \label{eqAAA}
A^{+[1]}:= \left(\bar{cccc}
\dun & \dzero & \dzero & \dzero \\
\dzero & \dzero & \dzero & \dzero \\
\dzero &  \dzero &   \dzero & \dzero\\
\dzero &  \dzero  & \dzero & \dzero\\
\dzero &  2 & \dzero & \dzero \\
\dzero &  \dzero  & \dzero & \dzero\\
\dzero &  \dzero  & \dzero & \dzero\\
\dzero &  \dzero  & \dzero & \dzero
\ear \right), \;
A^{-[1]}:= \left(\bar{cccc}
\dzero & 2 & \dzero & 2^{\otimes (-1)} \\
\dzero & \dzero & \dzero &  \dzero \\
\dzero & \dzero & \dzero &  \dzero \\
\dzero & \dzero & \dzero &  \dzero \\
\dzero & \dzero & \dzero &  \dun \\
\dzero & \dzero & \dzero &  \dzero \\
\dzero & \dzero & \dzero &  \dzero \\
\dzero & \dzero & \dzero &  \dzero
\ear \right).
\end{equation}
\end{subequations}

Here $\x=\{x_{1}\} \neq \emptyset$ and the $\textsf{mpr}$ problem is bounded because one can check eg.
that: $I_{\geq}(x_{1}) \neq \emptyset$. We have
$\x^{+}=\{x_{1}\}$ and of course $\x^{\dzero}=\emptyset$. The set of dominating variables is
$D=\emptyset$ (see (\ref{eqDomvar})). The condition (\ref{eqchok}) of Theorem~\ref{thmchattaignable} does not hold. \\

\noi
\textsc{Case $2.2.2$}, \S~\ref{subsubSubstit-cases} applies. We have (see \ref{eqdefID}):
$\I^{\geq}=I_{\geq}(x_{2})=\{(4,1)\}$. \\

Thus, obviously $(i^{*},j^{*})=(4,1)$ which means that $x_{1} = 2^{\otimes (-1)} \otimes h$,
in the row $4$ of the linear system of inequalities (\ref{eqAAA}). \\

We apply the update procedure of \S~\ref{subsecNewcost+Newconepmrp} and the new elements of
the problem are the following ones.

The set of linear equalities is now:

\begin{equation}
\mathcal{L}^{[2]}= \{x_{2} = 2 \otimes h, x_{1} = 2^{\otimes (-1)} \otimes h\}.
\end{equation}

The new cost function is:
\begin{equation}
  \textsf{cost}^{[2]}(x,h):= f_{z,32}(x,h)= h.
\end{equation}

The $4 \times 4$-transition matrix $T^{1 \rightarrow 2}$ is defined by:

\begin{equation}
  T^{1 \rightarrow 2}:= \left(\bar{cccc}
  \dun & \dzero & \dzero & \dzero \\
    \dzero & \dzero & \dzero & 2^{\otimes (-1)} \\  
  \dzero & \dzero & \dzero & \dzero \\
  \dzero & \dzero & \dzero & \dun 
  \ear \right). 
  \end{equation}

We calculate the following matrices $A^{+}:=A^{+[1]} \otimes T^{1
  \rightarrow 2}$ and $A^{-}:=A^{-[1]} \otimes T^{1 \rightarrow 2}$
and we obtain:
\[
A^{+}:= \left(\bar{cccc}
\dun & \dzero & \dzero & \dzero \\
\dzero & \dzero & \dzero & \dzero \\
\dzero &  \dzero &   \dzero & \dzero\\
\dzero &  \dzero  & \dzero & \dzero\\
\dzero &  \dzero & \dzero & \dun \\
\dzero &  \dzero  & \dzero & \dzero\\
\dzero &  \dzero  & \dzero & \dzero\\
\dzero &  \dzero  & \dzero & \dzero
\ear \right), \;
A^{-}:= \left(\bar{cccc}
\dzero & \dzero & \dzero & \dun \\
\dzero & \dzero & \dzero &  \dzero \\
\dzero & \dzero & \dzero &  \dzero \\
\dzero & \dzero & \dzero &  \dzero \\
\dzero & \dzero & \dzero &  \dun \\
\dzero & \dzero & \dzero &  \dzero \\
\dzero & \dzero & \dzero &  \dzero \\
\dzero & \dzero & \dzero &  \dzero
\ear \right).
\]

The
new cone $\mathcal{C}(A^{\; +}, A^{ \; -})^{[2]}$ is defined as follows:

\begin{subequations}

\begin{equation}
  \label{thelinsysofineq2}
A^{+[2]} \otimes w^{[2]} \geq  A^{-[2]} \otimes w^{[2]},
\end{equation}
  
with $w^{[2]}= (z,\dzero, \dzero, h)^{\intercal} \in \overline{\real}_{\dzero}^{4}$ and the matrices $(A^{\; +}, A^{ \; -})^{[2]}$
are defined as a result of $\textsf{setrowtozero}(A^{+}, A^{-})$:

\begin{equation}
A^{+[2]}:= \left(\bar{cccc}
\dun & \dzero & \dzero & \dzero \\
\dzero & \dzero & \dzero & \dzero \\
\dzero &  \dzero &   \dzero & \dzero\\
\dzero &  \dzero  & \dzero & \dzero\\
\dzero &  \dzero & \dzero & \dzero \\
\dzero &  \dzero  & \dzero & \dzero\\
\dzero &  \dzero  & \dzero & \dzero\\
\dzero &  \dzero  & \dzero & \dzero
\ear \right), \;
A^{-[2]}:= \left(\bar{cccc}
\dzero & \dzero & \dzero & \dun \\
\dzero & \dzero & \dzero &  \dzero \\
\dzero & \dzero & \dzero &  \dzero \\
\dzero & \dzero & \dzero &  \dzero \\
\dzero & \dzero & \dzero &  \dzero \\
\dzero & \dzero & \dzero &  \dzero \\
\dzero & \dzero & \dzero &  \dzero \\
\dzero & \dzero & \dzero &  \dzero
\ear \right).
  \end{equation}

\end{subequations}

We have: $\boldsymbol{-\infty}=A_{[|1,7|]h}^{+[2]} \geq A_{[|1,7|]h}^{-[2]} =\boldsymbol{-\infty}$. Thus, the
reachability condition (\ref{eqchok}) of Theorem~\ref{thmchattaignable} is verified. So, \textsf{mpr} has the
following solution:

\begin{equation}
z=h, \; x_{1} = 2^{\otimes (-1)} \otimes h, \; x_{2}=2 \otimes h.
\end{equation}
By taking $h=\dun$ we retrieve the solution of the Example 2 p. 1463 of \cite{kn:GKS012}.

\subsection{Fractional Example}
\label{subExfractional-de-Gaubertetal}
In this subsection we illustrate and compare our strongly polynomial method with the
pseudo-polynomial method applied on Example 3 p. 1469 of \cite{kn:GKS012}. The starting point of the 
pseudo-polynomial method is not a polyhedron but the homogenized cone of the following set
$\mathcal{P}(A,b,C,d) \cup \{\boldsymbol{\pm \infty}\}$ defined by:

\begin{equation}
A \otimes x \oplus b \geq C \otimes x \oplus d,
  \end{equation}
with
\begin{equation}
  A:=\left(\bar{ccc}
  \dzero & \dzero & \dzero \\
  \dzero & \dun & \dzero \\
  \dun & \dzero & \dzero \\
  \dun & \dzero & \dzero 
  \ear \right), \; b:=\left(\bar{c} \dun \\
  \dzero \\
  \dzero \\
  3
    \ear \right), \;  C:=\left(\bar{ccc}
  3^{\otimes (-1)} & 4^{\otimes (-1)} & \dzero \\
  1^{\otimes (-1)} & \dun & \dzero \\
  \dzero & \dzero & \dzero \\
  1 & \dzero & \dun 
  \ear \right), \; d:=\left(\bar{c} \dzero \\
  1 \\
  \dun \\
  \dzero
    \ear \right).
  \end{equation}
With the following change of notations $(A \; b) \leftrightarrow D$ and $(C \; d) \leftrightarrow C$. 

The fractional minimization problem \textsf{fmpr} we solve here is defined as:

\begin{equation}
\min\{z= x_{2} / 3 \otimes x_{1}, \; x=(x_{1}, x_{2}, x_{3})^{\intercal} \in \mathcal{P}(A,b,C,d)\}.
  \end{equation}

Applying the results of \S~\ref{seclinfraction}, the previous \textsf{fmpr} is transformed into the \textsf{mpr} problem defined by
(\ref{eqprobFractàresoudre})-(\ref{eqconefractprog})-(\ref{eqVartety})-(\ref{eqprobFractavec-h}) which is specified
hereafter.

\begin{equation}
  \bar{l}
  \min(z) \\
  z \geq \textsf{cost}^{[0]}(y,t,h):= y_{2}
  \ear
  \end{equation}
\begin{equation}
w^{[0]}:= (z, y_{1}, y_{2}, y_{3},t,h)^{\intercal} \in \mathcal{C}(A^{+}, A^{-})^{[0]}
\end{equation}

\begin{equation}
  (A^{+}, A^{-})^{[0]}:=\left(\bar{cccccc}
  \dun & \dzero & \dzero & \dzero & \dzero & \dzero \\
  \dzero & \dzero & \dzero & \dzero & \dun & \dzero \\
  \dzero & \dzero & \dun & \dzero & \dzero & \dzero \\
  \dzero & \dun & \dzero & \dzero & \dzero & \dzero \\
  \dzero & \dun & \dzero & \dzero & 3 & \dzero\\
        \dzero & 3 & \dzero & \dzero & \dzero & \dzero
        \ear\right), \; \left(\bar{cccccc}
        \dzero & \dzero & \dun & \dzero & \dzero & \dzero \\
        \dzero & 3^{\otimes (-1)} & 4^{\otimes (-1)} & \dzero & \dzero & \dzero \\
        \dzero & 1^{\otimes (-1)} & \dzero & \dzero & 1 & \dzero \\
                \dzero & \dzero & \dzero & \dzero & \dun & \dzero \\
                \dzero & 1 & \dzero & \dun & \dzero & \dzero \\
                                \dzero & \dzero & \dzero & \dzero & \dzero & \dun
        \ear\right).
  \end{equation}
The set $\{y_{1},y_{2}, y_{3},t\}$ is abbreviated by $(\y,t)$. 

The \textsf{mpr} is bounded with $I_{\geq}(y_{3}) = \emptyset$. $(\y,t)^{+}=\{y_{2}\}$
and $(\y,t)^{\dzero}=\{y_{1}, y_{3},t\}$. The dominating variables are in $D=\{y_{1}, y_{2}, t \}$ (see (\ref{eqDomvar})).
The condition (\ref{eqchok}) of Theorem~\ref{thmchattaignable} does not hold. \\

\noi
\textsc{Case $2.2.2$}, \S~\ref{subsubSubstit-cases} applies. We have (see \ref{eqdefID}):
$\I^{\geq}=\{(1,t), (2,2), (3,1),(4,t),(5,1) \}$. \\

\noi
We apply Theorem~\ref{theoNTZmin} with the set of valid inequalities $\I^{\geq}$.
Based on all the possibilities expressed in the next array using (\ref{a+a-ineq})-(\ref{ORa+a-ineq}):

\[
\bar{lllllll}
 \left( \bar{c} z \\ t
\ear \right)_{1} & \geq & \left( \bar{c} y_{2}\\
3^{\otimes (-1)} \otimes y_{1} \oplus 4^{\otimes (-1)} \otimes y_{2} \ear \right) & :: & \left(\bar{c}
\mbox{$[\dun \otimes \lambda, +\infty]$} \\ \mbox{$[ 3^{\otimes (-1)} \otimes \lambda,+\infty ]$} \ear \right) & :: & \left(\bar{c}
h\U \\  h\U \ear \right) \\
\left( \bar{c} z \\ y_{2}
\ear \right)_{2} & \geq & \left( \bar{c} 1^{\otimes (-1)} \otimes y_{1} \oplus 1 \otimes t\\
 1^{\otimes (-1)} \otimes y_{1} \oplus 1 \otimes t \ear \right) & :: & \left(\bar{c}
\mbox{$[1 \otimes \lambda, +\infty]$} \\ \mbox{$[1 \otimes \lambda,+\infty ]$} \ear \right) & :: & \left(\bar{c}
h\U \\  h\U \ear \right) \\
 \left( \bar{c} z \\ y_{1}
\ear \right)_{3} & \geq & \left( \bar{c} y_{2}\\
t \ear \right) & :: & \left(\bar{c}
\mbox{$[\dun \otimes \lambda, +\infty]$} \\ \mbox{$[\dun \otimes \lambda,+\infty ]$} \ear \right) & :: & \left(\bar{c}
h\U \\  h\U \ear \right) \\
\left( \bar{c} z \\ t
\ear \right)_{4} & \geq & \left( \bar{c} y_{2}\\
 2^{\otimes (-1)} \otimes y_{1} \oplus 3^{\otimes (-1)} \otimes y_{3} \ear \right) & :: & \left(\bar{c}
\mbox{$[\dun \otimes \lambda, +\infty]$} \\ \mbox{$[2^{\otimes (-1)} \otimes \lambda,+\infty ]$} \ear \right) & :: & \left(\bar{c}
h\U \\  h\U \ear \right) \\
\left( \bar{c} z \\ y_{1}
\ear \right)_{5} & \geq & \left( \bar{c} y_{2}\\
3^{\otimes (-1)} \otimes h \ear \right) & :: & \left(\bar{c}
\mbox{$[\dun \otimes \lambda, +\infty]$} \\ \mbox{$[3^{\otimes (-1)} \otimes \mu,+\infty ]$} \ear \right) & :: & \left(\bar{c}
h\U \\  h\B \ear \right) 
\ear
\]

we have: $\T=\U$ (see (\ref{eqdefTgeq})), $\mathcal{T}(\U,\I^{\geq})=\{(1,t), (2,2), (3,1),(4,t),(5,1) \}$
(see (\ref{mathcalTgeq})), $\textsf{argMin}(\theta)=\{(1,t), (3,1),(4,t),(5,1) \}$ with
  $\theta=[ \dun \otimes \lambda, +\infty]$ defined by (\ref{eqthetageq}). Because $\textsf{argMin}(\theta)$ is not
  a singleton we have: $\T'=\B$ (see (\ref{eqdefTT'geq})), $\mathcal{T}'(\B,\textsf{argMin}(\theta))=\{(5,1)\}$
  (see (\ref{mathcalTprimegeq}), $\textsf{argMin}(\theta')=\{(5,1)\}$ with $\theta'=[3^{\otimes (-1)} \otimes \mu,+\infty ]$
  with $\theta'$ defined by (\ref{eqthetaprimegeq}). Thus, \\
  $(i^{*},j^{*})=(5,1)$ and $y_{1}= 3^{\otimes (-1)} \otimes h $. 

  We apply the update procedure of \S~\ref{subsecNewcost+Newconepmrp} and the new elements of the problem are as
follows. \\

The set of linear equalities is now:

\begin{equation}
\mathcal{L}^{[1]}= \{y_{1}= 3^{\otimes (-1)} \otimes h\}.
\end{equation}

The new cost function is:
\begin{equation}
  \textsf{cost}^{[1]}(y,t,h):= f_{z,51}(y,t,h)=y_{2}.
\end{equation}

We define the
$6 \times 6$-transition matrix $T^{0 \rightarrow 1}$ as follows:

  \[
  T^{0 \rightarrow 1}:=\left(\bar{cccccc}
  \dun & \dzero & \dzero & \dzero & \dzero & \dzero \\
  \dzero & \dzero & \dzero & \dzero & \dzero & 3^{\otimes (-1)} \\
    \dzero & \dzero & \dun & \dzero & \dzero & \dzero \\
    \dzero & \dzero & \dzero & \dun & \dzero & \dzero \\
    \dzero & \dzero & \dzero & \dzero & \dun & \dzero \\
                \dzero & \dzero & \dzero & \dzero & \dzero & \dun
  \ear \right).
  \]

  We calculate the matrices $A^{+} = A^{+[0]} \otimes  T^{0 \rightarrow 1}$ and $A^{-} = A^{-[0]} \otimes  T^{0 \rightarrow 1}$ and it comes:

  \[
A^{+}=\left(\bar{cccccc}
  \dun & \dzero & \dzero & \dzero & \dzero & \dzero \\
  \dzero & \dzero & \dzero & \dzero & \dun & \dzero \\
  \dzero & \dzero & \dun & \dzero & \dzero & \dzero \\
  \dzero & \dzero & \dzero & \dzero & \dzero & 3^{\otimes (-1)} \\
  \dzero & \dzero & \dzero & \dzero & 3 & 3^{\otimes (-1)} \\
        \dzero & \dzero & \dzero & \dzero & \dzero & \dun
        \ear\right), \;
 A^{-}= \left(\bar{cccccc}
        \dzero & \dzero & \dun & \dzero & \dzero & \dzero \\
        \dzero & \dzero & 4^{\otimes (-1)} & \dzero & \dzero & 6^{\otimes (-1)} \\
        \dzero & \dzero & \dzero & \dzero & 1 & 4^{\otimes (-1)} \\
                \dzero & \dzero & \dzero & \dzero & \dun & \dzero \\
                \dzero & \dzero & \dzero & \dun & \dzero & 2^{\otimes (-1)}\\
                                \dzero & \dzero & \dzero & \dzero & \dzero & \dun
        \ear \right).      
        \]
We have $w^{[1]}=(z, \dzero,y_{2},y_{3},t,h)^{\intercal}$ and the matrices $(A^{\; +}, A^{ \; -})^{[1]}$
are defined as the result of $\textsf{setrowtozero}(A^{+}, A^{-})$ that is:

\begin{equation}
A^{+[1]}=\left(\bar{cccccc}
  \dun & \dzero & \dzero & \dzero & \dzero & \dzero \\
  \dzero & \dzero & \dzero & \dzero & \dun & \dzero \\
  \dzero & \dzero & \dun & \dzero & \dzero & \dzero \\
  \dzero & \dzero & \dzero & \dzero & \dzero & 3^{\otimes (-1)} \\
  \dzero & \dzero & \dzero & \dzero & 3 & 3^{\otimes (-1)} \\
        \dzero & \dzero & \dzero & \dzero & \dzero & \dzero
        \ear\right), \;
 A^{-[1]}= \left(\bar{cccccc}
        \dzero & \dzero & \dun & \dzero & \dzero & \dzero \\
        \dzero & \dzero & 4^{\otimes (-1)} & \dzero & \dzero & 6^{\otimes (-1)} \\
        \dzero & \dzero & \dzero & \dzero & 1 & 4^{\otimes (-1)} \\
                \dzero & \dzero & \dzero & \dzero & \dun & \dzero \\
                \dzero & \dzero & \dzero & \dun & \dzero & 2^{\otimes (-1)}\\
                                \dzero & \dzero & \dzero & \dzero & \dzero & \dzero
        \ear \right).  
  \end{equation}

We have $(\y,t)=\{y_{2}, y_{3},t\} \neq \emptyset$, $(\y,t)^{+}=\{y_{2}\}$ and $(\y,t)^{\dzero}=\{y_{3}, t\}$.
The \textsf{mpr} is bounded. Indeed, it is easy to check that
eg. $I_{\geq}(y_{2}) \neq \emptyset$, $I_{\leq}(y_{3}) \neq \emptyset$ and $I_{\geq}(t) \neq \emptyset$. The set of
dominating variables is $D= \{y_{2},t\}$ (see (\ref{eqDomvar})). The condition (\ref{eqchok}) of
Theorem~\ref{thmchattaignable} does not hold. \\

\noi
\textsc{Case $2.2.2$}, \S~\ref{subsubSubstit-cases} applies. We have
(see \ref{eqdefID}): $\I^{\geq}=\{(1,t), (2,2), (4,t)\}$. \\

We apply Theorem~\ref{theoNTZmin} with $\I^{\geq}=\{(1,t), (2,2), (4,t)\}$. \\

From the following array of possible inequalities provided using (\ref{a+a-ineq})-(\ref{ORa+a-ineq}):

\[
\bar{lllllll}
 \left( \bar{c} z \\ t
\ear \right)_{1} & \geq & \left( \bar{c} y_{2}\\
4^{\otimes (-1)} \otimes y_{2} \oplus 6^{\otimes (-1)} \otimes h\ear \right) & :: & \left(\bar{c}
\mbox{$[\dun \otimes \lambda, +\infty]$} \\ \mbox{$[ 6^{\otimes (-1)} \otimes \mu,+\infty ]$} \ear \right) & :: & \left(\bar{c}
h\U \\  h\B \ear \right) \\
\left( \bar{c} z \\ y_{2}
\ear \right)_{2} & \geq & \left( \bar{c} 1 \otimes t \oplus 4^{\otimes (-1)} \otimes h \\
 1 \otimes t \oplus 4^{\otimes (-1)} \otimes h\ear \right) & :: & \left(\bar{c}
\mbox{$[4^{\otimes (-1)} \otimes \mu, +\infty]$} \\ \mbox{$[4^{\otimes (-1)} \otimes \mu,+\infty ]$} \ear \right) & :: & \left(\bar{c}
h\B \\  h\B \ear \right) \\
\left( \bar{c} z \\ t
\ear \right)_{4} & \geq & \left( \bar{c} y_{2}\\
 3^{\otimes (-1)} \otimes y_{3} \oplus 5^{\otimes (-1)} \otimes h \ear \right) & :: & \left(\bar{c}
\mbox{$[\dun \otimes \lambda, +\infty]$} \\ \mbox{$[5^{\otimes (-1)} \otimes \mu,+\infty ]$} \ear \right) & :: & \left(\bar{c}
h\U \\  h\B \ear \right).
\ear
\]
we have $\T=\B$(see (\ref{eqdefTgeq})), $\mathcal{T}(\B,\I^{\geq})=\{(2,2)\}$ (see (\ref{mathcalTgeq})),
$\textsf{argMin}(\theta)=\{(2,2)\}$ with $\theta=[4^{\otimes (-1)} \otimes \mu, +\infty]$ defined by (\ref{eqthetageq}),
$S_{\theta}^{\geq}=\{y_{2}\}$ (see (\ref{Sthetageq})). $\textsf{argMin}(\theta)$ is a singleton thus: \\
$(i^{*},j^{*})=(2,2)$ and $y_{2}= 1 \otimes t \oplus 4^{\otimes (-1)} \otimes h$ \\

We apply the update procedure of \S~\ref{subsecNewcost+Newconepmrp} and the new elements of the problem are as
follows. \\

The set of linear equalities is now:

 \begin{equation}
   \mathcal{L}^{[2]} = \mathcal{L}^{[1]} \cup \{y_{2} = 1 \otimes t \oplus 4^{\otimes (-1)} \otimes h\} =
   \{y_{1} = 3^{\otimes (-1)} \otimes h, y_{2} = 1 \otimes t \oplus 4^{\otimes (-1)} \otimes h\}.
    \end{equation}

The new cost function is:
\begin{equation}
  \textsf{cost}^{[2]}(y,t,h)=  1 \otimes t \oplus 4^{\otimes (-1)} \otimes h.
  \end{equation}

We define the $6 \times 6$-transition matrix $T^{1 \rightarrow 2}$ which corresponds to
$y_{2}=1 \otimes t \oplus 4^{\otimes (-1)} \otimes h$ as follows:

  \[
  T^{1 \rightarrow 2}:=\left(\bar{cccccc}
  \dun & \dzero & \dzero & \dzero & \dzero & \dzero \\
  \dzero & \dzero & \dzero & \dzero & \dzero & \dzero \\
    \dzero & \dzero & \dzero & \dzero & 1 & 4^{\otimes (-1)} \\
    \dzero & \dzero & \dzero & \dun & \dzero & \dzero \\
    \dzero & \dzero & \dzero & \dzero & \dun & \dzero \\
                \dzero & \dzero & \dzero & \dzero & \dzero & \dun
  \ear \right).
  \]

  We calculate the matrices $A^{+} = A^{+[1]} \otimes  T^{1 \rightarrow 2}$ and $A^{-} = A^{-[1]} \otimes  T^{1 \rightarrow 2}$ and it comes:

  \[
  A^{+}=\left(\bar{cccccc}
  \dun & \dzero & \dzero & \dzero & \dzero & \dzero \\
  \dzero & \dzero & \dzero & \dzero & \dun & \dzero \\
  \dzero & \dzero & \dzero & \dzero & 1 &  4^{\otimes (-1)} \\
  \dzero & \dzero & \dzero & \dzero & \dzero & 3^{\otimes (-1)} \\
  \dzero & \dzero & \dzero & \dzero & 3 & 3^{\otimes (-1)} \\
        \dzero & \dzero & \dzero & \dzero & \dzero & \dzero
        \ear\right), \;
 A^{-}= \left(\bar{cccccc}
        \dzero & \dzero & \dun & \dzero & 1 & 4^{\otimes (-1)} \\
        \dzero & \dzero & \dzero & \dzero & 3^{\otimes (-1)}  & 6^{\otimes (-1)} \\
        \dzero & \dzero & \dzero & \dzero & 1 & 4^{\otimes (-1)} \\
                \dzero & \dzero & \dzero & \dzero & \dun & \dzero \\
                \dzero & \dzero & \dzero & \dun & \dzero & 2^{\otimes (-1)}\\
                                \dzero & \dzero & \dzero & \dzero & \dzero & \dzero
        \ear \right).  
\]

We have $w^{[2]}=(z, \dzero, \dzero,y_{3},t,h)^{\intercal}$ and the matrices $(A^{\; +}, A^{ \; -})^{[2]}$
are defined as the result of $\textsf{setrowtozero}(A^{+}, A^{-})$:

\begin{equation}
 A^{+[2]}=\left(\bar{cccccc}
  \dun & \dzero & \dzero & \dzero & \dzero & \dzero \\
  \dzero & \dzero & \dzero & \dzero & \dun & \dzero \\
  \dzero & \dzero & \dzero & \dzero & \dzero &  \dzero \\
  \dzero & \dzero & \dzero & \dzero & \dzero & 3^{\otimes (-1)} \\
  \dzero & \dzero & \dzero & \dzero & 3 & 3^{\otimes (-1)} \\
        \dzero & \dzero & \dzero & \dzero & \dzero & \dzero
        \ear\right), \;
 A^{-[2]}= \left(\bar{cccccc}
        \dzero & \dzero & \dun & \dzero & 1 & 4^{\otimes (-1)} \\
        \dzero & \dzero & \dzero & \dzero & 3^{\otimes (-1)}  & 6^{\otimes (-1)} \\
        \dzero & \dzero & \dzero & \dzero & \dzero & \dzero \\
                \dzero & \dzero & \dzero & \dzero & \dun & \dzero \\
                \dzero & \dzero & \dzero & \dun & \dzero & 2^{\otimes (-1)}\\
                                \dzero & \dzero & \dzero & \dzero & \dzero & \dzero
        \ear \right).  
\
\end{equation}

Now, $(\y,t)=\{y_{3}, t\} \neq \emptyset$ and the $\textsf{mpr}$ problem is bounded because one can check
eg. that: $I_{\leq}(y_{3}) \neq \emptyset$ and $I_{\geq}(t) \neq \emptyset$. We have
$(\y,t)^{+}=\{t\}$ and $(\y,t)^{\dzero}=\{y_{3}\}$. The set of dominating variables is
$D=\{t\}$ (see (\ref{eqDomvar})). Indeed, it is easy to see here that $I_{\geq}(y_{3})=\emptyset$ because
$A_{[|1,m|]y_{3}}^{+[0]} = \boldsymbol{-\infty}$. Clearly, the condition (\ref{eqchok}) of
Theorem~\ref{thmchattaignable} does not hold. \\

\noi
\textsc{Case $2.2.2$}, \S~\ref{subsubSubstit-cases} applies. We have (see \ref{eqdefID}):
$\I^{\geq}=I_{\geq}(t)=\{(1,t), (4,t)\}$. \\

We apply Theorem~\ref{theoNTZmin} with $\I^{\geq}=\{(1,t), (4,t)\}$. \\

From the following array of possible inequalities provided using (\ref{a+a-ineq})-(\ref{ORa+a-ineq}):

\[
\bar{lllllll}
 \left( \bar{c} z \\ t
\ear \right)_{1} & \geq & \left( \bar{c}  4^{\otimes (-1)} \otimes h \\
 6^{\otimes (-1)} \otimes h \ear \right) & :: & \left(\bar{c}
\mbox{$[4^{\otimes (-1)} \otimes \mu, +\infty]$} \\ \mbox{$[6^{\otimes (-1)} \otimes \mu,+\infty ]$} \ear \right) & :: & \left(\bar{c}
  h\B \\  h\B \ear \right) \\
\left( \bar{c} z \\ t
\ear \right)_{4} & \geq & \left( \bar{c} 3^{\otimes (-1)} \otimes y_{3} \oplus 4^{\otimes (-1)} \otimes h \\
3^{\otimes (-1)} \otimes y_{3} \oplus  5^{\otimes (-1)} \otimes h \ear \right) & :: & \left(\bar{c}
    \mbox{$[4^{\otimes (-1)} \otimes \mu, +\infty]$} \\
    \mbox{$[5^{\otimes (-1)} \otimes \mu,+\infty ]$} \ear \right) & :: & \left(\bar{c}
  h\B \\  h\B \ear \right).
\ear
\]
one has $\T= \B$, (see (\ref{eqdefTgeq})), $ \mathcal{T}(\B,\I^{\geq})=\{(1,t),(4,t)\}$ (see (\ref{mathcalTgeq})),
$\textsf{argMin}(\theta)=\{(1,t),(4,t)\}$ with $\theta=[4^{\otimes (-1)} \otimes \mu, +\infty]$ defined by (\ref{eqthetageq}).
$\textsf{argMin}(\theta)$ is not a singleton. So, we have $\T'=\B$ (see (\ref{eqdefTT'geq})),
$\mathcal{T}'(\B,\textsf{argMin}(\theta))=\{(1,t),(4,t)\}$ (see (\ref{mathcalTprimegeq})),
$\textsf{argMin}(\theta')=\{(4,t)\}$ with $\theta'=[5^{\otimes (-1)} \otimes \mu,+\infty ]$ defined by (\ref{eqthetaprimegeq}).
Thus, \\
$(i^{*}, j^{*}) \in \textsf{argMin}(\theta')=\{(4,t)\}$ and $t=3^{\otimes (-1)} \otimes y_{3} \oplus 5^{\otimes(-1)}\otimes h$. \\

We apply the update procedure of \S~\ref{subsecNewcost+Newconepmrp}. \\

The set of linear equalities is now:

\begin{equation}
  \mathcal{L}^{[3]} = \{y_{1} = 3^{\otimes (-1)} \otimes h, y_{2} = 1 \otimes t \oplus 4^{\otimes (-1)} \otimes h\} \cup
    \{t=3^{\otimes (-1)} \otimes y_{3} \oplus  5^{\otimes (-1)} \otimes h\}.
\end{equation}

And the new cost function is:
\begin{equation}
\textsf{cost}^{[3]}(y,t,h)= 3^{\otimes (-1)} \otimes y_{3} \oplus 4^{\otimes (-1)} \otimes h.
  \end{equation}

We define the $6 \times 6$-transition matrix $T^{1 \rightarrow 2}$ as follows:

  \[
  T^{2 \rightarrow 3}:=\left(\bar{cccccc}
  \dun & \dzero & \dzero & \dzero & \dzero & \dzero \\
  \dzero & \dun & \dzero & \dzero & \dzero & \dzero \\
    \dzero & \dzero & \dun & \dzero & \dzero & \dzero \\
    \dzero & \dzero & \dzero & \dun & \dzero & \dzero \\
    \dzero & \dzero & \dzero & 3^{\otimes (-1)} & \dzero & 5^{\otimes (-1)} \\
                \dzero & \dzero & \dzero & \dzero & \dzero & \dun
  \ear \right).
  \]
 We calculate the matrices $A^{+} = A^{+[2]} \otimes  T^{2 \rightarrow 3}$ and $A^{-} = A^{-[2]} \otimes  T^{2 \rightarrow 3}$ and we have:

 \[
 A^{+}=\left(\bar{cccccc}
  \dun & \dzero & \dzero & \dzero & \dzero & \dzero \\
  \dzero & \dzero & \dzero & 3^{\otimes (-1)} & \dzero & 5^{\otimes (-1)} \\
  \dzero & \dzero & \dzero & \dzero & \dzero &  \dzero \\
  \dzero & \dzero & \dzero & \dzero & \dzero & 3^{\otimes (-1)} \\
  \dzero & \dzero & \dzero & \dun & \dzero & 2^{\otimes (-1)} \\
        \dzero & \dzero & \dzero & \dzero & \dzero & \dzero
        \ear\right), \;
 A^{-}= \left(\bar{cccccc}
        \dzero & \dzero & \dzero & 3^{\otimes (-1)} & \dzero & 4^{\otimes (-1)} \\
        \dzero & \dzero & \dzero & 6^{\otimes (-1)} & \dzero  & 6^{\otimes (-1)} \\
        \dzero & \dzero & \dzero & \dzero & \dzero & \dzero \\
                \dzero & \dzero & \dzero & 3^{\otimes (-1)} & \dzero & 5^{\otimes (-1)}\\
                \dzero & \dzero & \dzero & \dun & \dzero & 2^{\otimes (-1)}\\
                                \dzero & \dzero & \dzero & \dzero & \dzero & \dzero
                                \ear \right).
\]
We have $w^{[3]}=(z, \dzero, \dzero,y_{3}, \dzero,h)^{\intercal}$ and the matrices $(A^{\; +}, A^{ \; -})^{[3]}$
are defined as the result of $\textsf{setrowtozero}(A^{+}, A^{-})$:

\begin{equation}
A^{+[3]}=\left(\bar{cccccc}
  \dun & \dzero & \dzero & \dzero & \dzero & \dzero \\
  \dzero & \dzero & \dzero & \dzero & \dzero & \dzero \\
  \dzero & \dzero & \dzero & \dzero & \dzero &  \dzero \\
  \dzero & \dzero & \dzero & \dzero & \dzero & 3^{\otimes (-1)} \\
  \dzero & \dzero & \dzero & \dzero & \dzero & \dzero \\
        \dzero & \dzero & \dzero & \dzero & \dzero & \dzero
        \ear\right), \;
 A^{-[3]}= \left(\bar{cccccc}
        \dzero & \dzero & \dzero & 3^{\otimes (-1)} & \dzero & 4^{\otimes (-1)} \\
        \dzero & \dzero & \dzero & \dzero & \dzero  & \dzero \\
        \dzero & \dzero & \dzero & \dzero & \dzero & \dzero \\
                \dzero & \dzero & \dzero & 3^{\otimes (-1)} & \dzero & 5^{\otimes (-1)}\\
                \dzero & \dzero & \dzero & \dzero & \dzero & \dzero \\
                                \dzero & \dzero & \dzero & \dzero & \dzero & \dzero
                                \ear \right).
  \end{equation}

We have:
\[
A_{[|1,5|]h}^{+[3]}= \left(\bar{c}
\dzero \\
\dzero \\
3^{\otimes (-1)} \\
\dzero \\
\dzero
\ear \right) \geq A_{[|1,5|]h}^{-[3]} = \left(\bar{c}
\dzero \\
\dzero \\
5^{\otimes (-1)} \\
\dzero \\
\dzero
\ear \right),
\]
and the condition (\ref{eqchok}) of Theorem~\ref{thmchattaignable} is verified. Then, 
the function $h \mapsto 4^{\otimes (-1)} \otimes h$ of the cost function 
$\textsf{cost}^{[3]} (y,t,h)= 3^{\otimes (-1)} \otimes y_{3} \oplus 4^{\otimes (-1)} \otimes h$
is reached at $y_{3} = \dzero$. Thus, we have to solve the following triangular linear system of equalities
$\mathcal{L}^{[4]}$:

\begin{equation}
\bar{l}
\mathcal{L}^{[4]}: \\
z=3^{\otimes (-1)} \otimes y_{3} \oplus 4^{\otimes (-1)} \otimes h \\
y_{1} = 3^{\otimes (-1)} \otimes h \\
  y_{2} = 1 \otimes t \oplus 4^{\otimes (-1)} \otimes h \\
  t=3^{\otimes (-1)} \otimes y_{3} \oplus  5^{\otimes (-1)} \otimes h \\
  y_{3}= \dzero.
  \ear
\end{equation}

And we have the solution of the \textsf{mpr} obtained by obvious backward substitution in $\mathcal{L}^{[4]}$:
\begin{equation}
  \min(z)= 4^{\otimes (-1)} \otimes h, \;
  y=\left(\bar{c} 3^{\otimes (-1)} \otimes h \\ 4^{\otimes (-1)} \otimes h \\ \dzero \ear \right)
    , \; t= 5^{\otimes (-1)} \otimes h.
\end{equation}

If we take $h=\dun$ then the solution $x$ to the initial \textsf{fmpr} is: $x= t^{\otimes (-1)} \otimes y = \left(\bar{c} 2 \\ 1 \\ \dzero \ear \right)$
and $x_{2} / 3 \otimes x_{1} = 1 \otimes (3 \otimes 2)^{\otimes (-1)} = 1 \otimes 5^{\otimes (-1)} = 4^{\otimes (-1)}$.
We retrieve the minimum ($(-4)$ with the usual notations for reals) found
by the pseudo-polynomial method developed in \cite{kn:GKS012} and applied on Example 3 p. 1469 of \cite{kn:GKS012}.

\section{Conclusion}
\label{secConcl}

In this conclusion we try to provide a more precise analysis of the complexity of the
substitution method. We also compare our result with other known results. \\

For the forward substitution we have to establish the hierarchy between the variables of the problem
we need to study $m$ inequalities
of the form (\ref{a+a-ineq}) which generate at most $n$ \textsf{AND}-inequalities (\ref{ANDa+a-ineq}) and
at most $n$ \textsf{OR}-inequalities (\ref{ORa+a-ineq}): $\mathcal{O}(2nm)$. From the \textsf{OR}-inequalities
the computation complexity of the functions $f^{\geq}_{ij}$
(\ref{deffij})-(\ref{deflij})-(\ref{eqdefuij})-(\ref{eqdefrij}) is $\mathcal{O}(nm)$. From the
$f^{\geq}_{ij}$ the computation complexity of the $f_{z,ij}$ functions (the possible next
cost function (\ref{fzij2})) is also  $\mathcal{O}(nm)$. To decide which variable can be substituted
we need Theorem~\ref{theoNTZmin} and compute the intervals $f_{z,ij}(R_{\lambda \mu}(x,h))$ and
$f_{ij}(R_{\lambda \mu}(x,h))$ in $\mathcal{O}(2nm)$. This procedure is repeated $n$ times and the overall complexity of the
forward susbtitution is thus: 
\[
\mathcal{O}(6n^{2}m).
\]
Now we have to add the time complexity of the resolution of the triangular system of equalities
$\mathcal{L}^{[n]}$ which is $\mathcal{O}(n)$ because we do not need to complute a
pseudo inverse of a matrix by residuation theory (see eg. \cite{kn:Bac-cooq}). This is
only backward substitution. Finally, the update procedure in $\mathcal{O}(mn^{2})$ (see \S~\ref{subsecNewcost+Newconepmrp})
is applied $n$ times thus, the
whole complexity is:
\begin{equation}
  \label{eqComplexiteTotale}
\mathcal{O}(6n^{2}m) + \mathcal{O}(n) + \mathcal{O}(mn^{3}) \approx \mathcal{O}(mn^{3}).
\end{equation}

Let us recall that to the best of our knowledge this substitution method is the first strongly polynomial method
for such problem. All other proposed schemes of resolution are pseudo-polynomial (see \cite{kn:BA08}, \cite{kn:GKS012},
\cite{kn:GMH014}). Thus, because $(\max,+)$-linear programming is now proved to be strongly polynomial
then the mean payoff games problem is also strongly polynomial (see eg. \cite{kn:GKS012}). Let us
remark that this result was already proved in \cite{truffet:hal-05491586} (english paper)
when solving the max-atom problem (MAP).
The important consequence of this result is that other six PTIME equivalent
problems to \textsf{MAP} which were known to be in NP $\cap$ co-NP
(see eg \cite{Akianetal2010}, \cite{Bezemetal2010}, \cite{Condon1992},
\cite{Ehren1979}, \cite{Galloetal1993}, \cite{Gurvichetal1988},
\cite{Mine2001}, \cite{Mohringetal2004}, \cite{Mieuw2006}, \cite{Zwick1996}) are also
strongly polynonial. We list them hereafter. P1: Looking for non trivial
  solutions of a tropical cone, P2: Computation of a tropical rank
  of a matrix, P3: Computation of optimal strategies in parity games
  (typically: Mean Payoff Games), P4: Scheduling with and/or
  precedence constraints, P5: Shortest path problem in hypergraph,
  P6: Model checking and $\mu$-calculus. \\
  And we repeat that the best pseudo-polynomial method known till now
  can be replaced by our strongly polynomial approach.

\bibliographystyle{plain}
\bibliography{cone9th}

\appendix

\section{Maximizing a tropical linear problem}
\label{secMPR}
In this appendix we follow the same organization of the main part of the paper dealing
with the minimization problem. All the results below are listed without proofs. \\

To formulate the maximization problem $\textsf{MPR}$ first we borrow the Fourier's trick
(see eg. \cite{kn:Juhel2025}, \cite{kn:Williams86}) used
in the usual linear algebra and applying it to the $(\max,+)$-algebra. It comes
that $\max(z=c^{\intercal} \otimes x \oplus c_{h} \otimes h)$ is replaced by:
$\max(z) \mbox{ and } z \leq c^{\intercal} \otimes x \oplus c_{h} \otimes h$. And then, we will solve the following
minimization problem $\textsf{MPR}$ defined as:

\begin{subequations}
  \begin{equation}
    \label{maxz}
\bar{l}
\max(z)
\ear
    \end{equation}
  such that:
  \begin{equation}
(x,h) \in \overline{\real}_{\dzero}^{n+1},
  \end{equation}
  and
 \begin{equation}
   \label{mprContraintesmax}
   \bar{l}
    z \leq c^{\intercal} \otimes x \oplus c_{h} \otimes h \\
    A \otimes x \oplus b \otimes h \leq C \otimes x \oplus d \otimes h.
  \ear
\end{equation}
 Where the homogenization variable $h$ satisfies the following condition:
 \begin{equation}
   \label{cond-h-different-infinimax}
h < +\infty.
 \end{equation}
And
 \begin{equation}
\mbox{$z$ and $h$ are not substituable}.
 \end{equation}
 Finally, among the variables of the problem $z, x_{1}, \ldots, x_{n}$ we have:
\begin{equation}
  \label{zpriomax}
  \mbox{$z$ has the highest priority because bounded by the cost function}.
  \end{equation}

\end{subequations}

When necessary such problem will be denoted:
\\
$\textsf{MPR}(A,b,C,d,c,c_{h},z,x,h,(m,n))$. \\

To the cone $\mathcal{C}(L,W):=\{x: L \otimes x \leq W \otimes x\}$
we associate the function $\overline{\textsf{setrowtozero}}(L,W)$ defined by:

  \begin{equation}
    \label{ligneazeromax}
    \bar{lll}
    \overline{\textsf{setrowtozero}}(L,W) & := & \mbox{ For $i=1$ to $m$ do} \\
    \mbox{ } & \mbox{} & \mbox{ if $l_{i,.} \leq w_{i,.}$ then $l_{i,.}:=\dzero,  w_{i,.}:=\dzero$}
    \ear
  \end{equation}

For the \textsf{MPR} we will consider the following familly of intervals:

\begin{equation}
  \label{defIntervalu+infinimax}
\mathcal{J}:=\{[\dzero, u], u \in \real_{\dzero}\}.
  \end{equation}

\begin{propo}
  \label{propMinIntervalmax}
  Let $[\dzero,u_{l}]$, $l=1, \ldots,k$, $k \geq 2$ be $k$ intervals of $\mathcal{J}$ with
For all $u_{k} \leq \ldots \leq u_{1}$, we have the following intervals inclusion:
  \begin{equation}
[\dzero,u_{k}] \subseteq [\dzero, u_{k-1}] \subseteq \cdots \subseteq [\dzero,u_{1}].
    \end{equation}
  \end{propo}
If $\overline{\textsf{Min}}$ denotes the minimum in the sense of interval inclusion then:

\begin{equation}
  \label{eq-intersection=Min-intervalmax}
\cap_{l=1}^{k} [\dzero,u_{l}]=[\dzero, u_{k}] = \overline{\textsf{Min}}\{[\dzero , u_{l}], l=1, \ldots,k\}.
  \end{equation}

At the step $0$ of the substitution the cone associated with the
constraints system (\ref{mprContraintes}) of the $\textsf{MPR}$
problem is denoted $\mathcal{C}(A^{\; +}, A^{ \; -})^{[0]}$ and is
defined by:

\begin{subequations}
\begin{equation}
  \mathcal{C}(A^{+}, A^{-})^{[0]}:=\{w \in \overline{\real}_{\dzero}^{n+2}: A^{+ [0]} \otimes w \leq
  A^{-[0]} \otimes w \}
  \end{equation}
where:

\begin{equation}
  (A^{+}, A^{-}, w)^{[0]}:=\left(\bar{ccc}
  \dun & \boldsymbol{-\infty}^{\intercal} & \dzero \\
  \boldsymbol{-\infty} & A & b
  \ear\right), \; \left(\bar{ccc}
  \dzero & c^{\intercal} & c_{h} \\
  \boldsymbol{-\infty} & C & d
  \ear\right), \; \left(\bar{c}
z \\ x \\ h
  \ear\right).
  \end{equation}
  \end{subequations}

\noi
We also use the same numerotation convention~\ref{numMatcol} p. \pageref{numMatcol} for matrices and
vectors involved in this problem.

The cost function at step $0$ of the substitution is denoted and defined by: $\textsf{cost}^{[0]}(x,h):= c^{\intercal} \otimes x \oplus c_{h} \otimes h$,
with $c_{h}=\dzero$. \\

The set of the stored linear equalities is denoted $\mathcal{L}^{[k]}$ at each step $k$ of the
substitution. And for $k=0$ we have $\mathcal{L}^{[0]}= \emptyset$. \\

The vector $x$ is called the vector of the remaining variables of the \textsf{MPR}. Of course at
step $0$ of the method: $x=(x_{j})_{j=1}^{n}$. To $x$ we associate the set of the remaining variables
denoted $\x$.

The parametrized domain of research of a maximum is denoted $\tilde{R}_{\tilde{\lambda} \tilde{\mu}}(x,h)$ and defined by:

\begin{subequations}
\begin{equation}
  \label{eqRxhmax}
\tilde{R}_{\tilde{\lambda} \tilde{\mu}}(x,h):= \times_{x_{j} \in \x} [\dzero, \tilde{\lambda}] \times [\dzero, \tilde{\mu}].
  \end{equation}
Where because $h$ must be $ < +\infty$ we can assume that the following condition holds:

\begin{equation}
  \label{eqmudominelambdamax}
  \forall \theta \neq \dzero, \forall \beta, \;  \exists \tilde{\lambda}, \tilde{\mu}: \;
 \beta \otimes \tilde{\mu} < \theta \otimes \tilde{\lambda}.
\end{equation}
\end{subequations}

Based on the result of Proposition~\ref{propineqValid}
$\forall i=1, \ldots,m$, $\forall j=1, \ldots, n$ such that the following condition is
satisfied:

\begin{equation}
  \label{condExistgeqmax}
  \mbox{$a^{+}_{i,j} \neq \dzero$ and
    $a^{+}_{i,j} > a^{-}_{i,j}$}
\end{equation}

we define the
following valid inequality:

\begin{equation}
  \label{ineqIminmax}
I_{\leq}(x_{j},i):=\{a^{+}_{i,j} \otimes x_{j} \leq a^{-}_{i,.} \otimes w\}.
  \end{equation}
Under condition (\ref{condExistgeqmax}) the inequality $I_{\leq}(x_{j},i)$ can be rewritten as:

    \begin{equation}
I_{\leq}(x_{j},i)=\{x_{j} \leq f^{\leq}_{ij}(x,h)\}.
      \end{equation}

Where the $(x,h)$-linear function $f^{\leq}_{ij}(x,h)$ is defined by:

    \begin{subequations}
      \begin{equation}
        \label{deffijmax}
f^{\leq}_{ij}(x,h):= \ell_{ij}(x) \oplus r_{ij} \otimes h, 
    \end{equation}
      with $\ell_{ij}$ is the following $x$-linear function which does not
      depend on $x_{j}$:
    \begin{equation}
      \label{deflijmax}
\ell_{ij}(x):= v^{\intercal}_{ij} \otimes x,
    \end{equation}
    where $v_{ij}$ denotes the  $n$-dimensional vector such that:
    \begin{equation}
      \label{eqdefuijmax}
      \forall j'=1, \ldots, n: \; v_{ij,j'}=\left\{\bar{cc}
      (a^{+}_{i,j})^{\otimes (-1)} \otimes a^{-}_{i,j'} & \mbox{if $j' \neq j$} \\
      \dzero & \mbox{if $j'=j$}
      \ear \right.
    \end{equation}
   and the scalar $r_{ij}$ is defined by:
    \begin{equation}
      \label{eqdefrijmax}
r_{ij}:= (a^{+}_{i,j})^{\otimes (-1)} \otimes d_{i}.
      \end{equation}
    \end{subequations}

    Assuming that $I_{\leq}(x_{j},i) \neq \emptyset$, or equivalently be a  valid inequality, we define the following function:

    \begin{equation}
      \label{fzijmax}
f_{z,ij}(x,h):= \oplus_{j' \neq j} c_{j'} \otimes x_{j'} \oplus c_{j} \otimes f^{\leq}_{ij}(x,h) \oplus c_{h} \otimes h.
      \end{equation}
    By replacing $f^{\leq}_{ij}(x,h)$ in (\ref{fzij}) we obtain the new expression for $f_{z,ij}(x,h)$:

    \begin{equation}
      \label{fzij2max}
f_{z,ij}(x,h):= \oplus_{j' \neq j} (c_{j'} \oplus c_{j} \otimes v_{ij,j'}) \otimes x_{j'} \oplus (c_{h} \oplus c_{j} \otimes r_{ij}) \otimes h.
      \end{equation}

The next Theorem will be useful to justify the substitution method developed in this paper.

    \begin{theo}[Saturation of an inequality]
      \label{theoMinCimax}
      For all $x_{j}$ such that $I_{\leq}(x_{j},i)$ is valid in the sense of Proposition~\ref{propineqValid} we have:

      \begin{equation}
        \label{infimumzxjmax}
        \sup_{\{x: \; x_{j} \leq f^{\leq}_{ij}(x,h)\}} \left(\bar{c} c^{\intercal} \otimes x \oplus c_{h} \otimes h \\
        x_{j} \ear \right) \; = \left(\bar{c} f_{z,ij}(x,h) \\
        f^{\leq}_{ij}(x,h)
        \ear \right).
        \end{equation}
      And the supremum is reached at $x$ such that we have:
      \begin{equation}
        \label{substitxjmax}
x_{j} = f^{\leq}_{ij}(x,h).
        \end{equation}
      
    \end{theo}

\subsection{Classification of the variables and the linear functions of \textsf{MPR}}

    Based on the result of Proposition~\ref{propineqValid},
    $\forall i=1, \ldots,m$, $\forall j=1, \ldots, n$ such that the following condition is
satisfied:

\begin{equation}
  \label{condExistleqmax}
\mbox{$a^{-}_{i,j} \neq \dzero$ and
  $a^{+}_{i,j} < a^{-}_{i,j}$}
\end{equation}

we define the following valid inequality:

\begin{equation}
  \label{ineqIminminmax}
I_{\geq}(x_{j},i):=\{a^{-}_{i,j} \otimes x_{j} \geq a^{+}_{i,.} \otimes w\}.
  \end{equation}
Under condition (\ref{condExistleq}) the result of Proposition~\ref{propineqValid}
$I_{\geq}(x_{j},i) = \{x_{j} \geq f^{\geq}_{ij}(x,h)\}$ (see (\ref{deffij})-(\ref{eqdefrij})).

The hierarchy between the variables of the problem is denoted $\preceq_{\overline{var}}$ and
based on their variation domain such that:

\begin{equation}
  \label{eqOrdreVarmax}
  \{x_{j} \in \x: \mbox{$I_{\leq}(x_{j})\neq \emptyset$ and $I_{\geq}(x_{j}) \neq \emptyset$}\} \preceq_{\overline{var}}
  \{x_{j} \in \x: \mbox{$I_{\leq}(x_{j})\neq \emptyset$ and $I_{\geq}(x_{j}) =\emptyset$}\}.
\end{equation}

We also mentioned that: \\
$\{x_{j} \in \x: \mbox{$I_{\leq}(x_{j})\neq \emptyset$ and $I_{\geq}(x_{j}) =\emptyset$}\}
  \preceq_{\overline{var}}
  \{x_{j} \in \x: \mbox{$I_{\leq}(x_{j}) = \emptyset$ and $I_{\geq}(x_{j}) =\emptyset$}\}$.

We now define the following partition of non null linear functions involved in the \textsf{MPR}.
  
  \begin{defi}
    \label{defhBhUmax}
  A non null $(x,h)$-linear function $f: (x,h) \mapsto \alpha^{\intercal} \otimes x \oplus \beta \otimes h$ is said to be
  $h$-bounded if $\alpha = \boldsymbol{-\infty}$. Otherwise the function $f$ is
  said to be $h$-unbounded.
  \end{defi}

We denote $h\tilde{\B}$ the set of $h$-bounded non null linear functions and we denote $h\tilde{\U}$ the set of
non null $h$-unbounded linear functions. And $(h\tilde{\B}, h\tilde{\U})$ is a partition of the set of all non
null $(x,h)$-linear functions. We have the following interval calculus results.

\bit

\item $\forall f= \alpha^{\intercal} \otimes x \oplus \beta \otimes h \in h\tilde{\U}$, using interval calculus formulae we have:
\begin{subequations}
  \begin{equation}
    \label{eqfRdomainemax}
  f(\tilde{R}_{\tilde{\lambda} \tilde{\mu}}(x,h)) =[\dzero, \oplus_{i} \alpha_{i} \otimes \tilde{\lambda} \oplus \beta \otimes \tilde{\mu}],
  \end{equation}
  and by assumption ((\ref{eqmudominelambdamax}) with $\theta= \oplus_{i} \alpha_{i}$) on $\tilde{\lambda}$ and $\tilde{\mu}$ we have:
  
  \begin{equation}
    \label{eqDomainhBmax}
    \forall \alpha \in \real_{\dzero}^{n} \setminus \{\boldsymbol{-\infty}\}, \forall \beta: \; f(\tilde{R}_{\tilde{\lambda} \tilde{\mu}}(x,h)) =
            [\dzero,  \oplus_{i} \alpha_{i} \otimes \tilde{\lambda}]
    \end{equation}
  Let us stress that this interval does not depend on  $\beta$ and
  by assumption (\ref{eqmudominelambdamax}):
  \begin{equation}
    \label{eqDominclusmax}
   [\dzero, \beta \otimes \tilde{\mu}] \subset [\dzero, \oplus_{i} \alpha_{i} \otimes \tilde{\lambda}].
    \end{equation}
\end{subequations}

  \item  And $\forall f= \beta \otimes h \in h\tilde{\B}$:

    \begin{equation}
    f(\tilde{R}_{\tilde{\lambda} \tilde{\mu}}(x,h)) =[\dzero, \beta \otimes \tilde{\mu}].
    \end{equation}
    \eit

From this above results we have the following
hierarchy between the non null linear functions denoted $\preceq_{\overline{fct}}$ which is based on their
value domain:

\begin{equation}
  \label{eqOrdrefctmax}
\{f: f \in  h\tilde{\B} \} \preceq_{\overline{fct}} \{f: f \in h\tilde{\U} \}.
  \end{equation}

\subsection{Choosing a variable for a new substitution after $k$ substitutions in \textsf{MPR}}
\label{subsubChoixvarmax}

The procedure for choosing a remaining variable to be
substituted is based on the theorems listed below. Recall that we use numerotation convention~\ref{numMatcol}
p. \pageref{numMatcol}.

Because when $\x^{+} \neq \emptyset$ the cost function can increases we have the following result.

\begin{theo}[Maximality and reachability at $\boldsymbol{-\infty}$ of $h \mapsto c_{h} \otimes h$]
  \label{thmchattaignablemax}
  The function $h \mapsto c_{h} \otimes h$ is the
  lower bound of the cost function $(x,h) \mapsto c^{\intercal}_{\x} \otimes x \oplus c_{h} \otimes h$ which is attained
  at $x=\boldsymbol{-\infty}$ iff the following conditions hold:
  \begin{subequations}
    the cost function cannot increase, ie.:
  \begin{equation}
    \label{eqchokmax0}
    \x^{+} = \emptyset,
  \end{equation}
  and the homogenization variable $h$ can take any arbitrary value, ie.:
  \begin{equation}
    \label{eqchokmax}
A_{[|1,m|] h}^{+[k]} \leq A_{[|1,m|] h}^{-[k]}.
    \end{equation}
\end{subequations}
  \end{theo}

Before stating the main Theorem of this appendix we need the following preliminary Lemma.

\begin{lemm}
      \label{lemmInegalitesgeqgeq}
      Let us consider two tropical $(x,h)$-linear functions $f_{1}$ and $f_{2}$, a variable $y$ and the
      two inequalities $I_{\leq}(y,1):=(y \leq f_{1}(x,h))$ and $I_{\leq}(y,2):=(y \leq f_{2}(x,h))$. We assume that
      $(x,h)$ varies in some domain and that 
      the image of the domain by $f_{i}$ is the interval $\sigma_{i}:=[\lambda_{i},+\infty]$, $i =1,2$. 

      W. l. o. g. we can assume that $\lambda_{1} \leq \lambda_{2}$. Then, we have the following result. \\

The maximum interval in the sense of the inclusion denoted $\textsf{Sup}_{\sigma \in \mathcal{J}}
  (\forall y \in \sigma \; \exists (x,h) \;
  (y \leq f_{1}(x,h) \mbox{ and } (y \leq f_{2}(x,h)) \mbox{ is true})$ is equal to
            $\overline{\textsf{Min}}(\sigma_{1}, \sigma_{2})$.

\end{lemm}
\proof We just have to remark that $\forall y \in \sigma \; \exists (x,h) \; (y \leq f_{1}(x,h) \mbox{ and } (y \leq f_{2}(x,h)) \mbox{ is true})$
for $\sigma \in \mathcal{J}$ iff
$\sigma \subseteq [\dzero,\lambda_{1}]=\overline{\textsf{Min}}(\sigma_{1} , \sigma_{2})$.

\cqfd \\

\begin{theo}
  \label{theoNTZminmax}
  Let $\I^{\leq}$ be a subset of valid inequalities $I_{\leq}(x_{j},i)$, ie. a set
such that:
\begin{equation}
  \I^{\leq} \subseteq I_{\leq}(\x).
\end{equation}
Where
\begin{equation}
  \label{eqtteslesgeqmax}
I_{\leq}(\x):= \{(i,j): \mbox{ $x_{j} \in \x$ and $I_{\leq}(x_{j},i)$ is valid} \}.
  \end{equation}

\begin{equation}
  \label{implicationthmNTZmax}
\mbox{ $I_{\leq}(x_{j},i)\neq \emptyset$ or equivalently $I_{\leq}(x_{j},i)$ is valid} \Rightarrow z \leq f_{z,ij}(x,h).
  \end{equation}
Recalling that $f_{z,ij}(x,h)$ is defined by (\ref{fzij2max}). \\

Let $\T$ and $\T'$ denoting either $\tilde{\B}$ or $\tilde{\U}$, respectively.
The $2$-tuple $(\T,\T')$ is a function of $\I$ and is defined according
to $\preceq_{fct}$ as follows with priority to the cost function.
First, let us define $\T$ as follows:

  \begin{equation}
    \label{eqdefTleq}
    \T:=\left\{\bar{ll} \tilde{\B} & \mbox{ if $\exists (i,j) \in \I^{\leq}$ s.t. $f_{z,ij} \in h\tilde{\B}$} \\
    \tilde{\U} & \mbox{ otherwise}.
    \ear \right.
    \end{equation}
  
Then, let us define the following set:

\begin{equation}
  \label{mathcalTleq}
  \mathcal{T}(\T,\I^{\leq}):=\{(i,j) \in \I^{\leq}: f_{z,ij} \in h\T \}.
  \end{equation}

Let us define the logical decomposition of the set $\mathcal{T}(\T,\I^{\leq})$ as follows.

     \begin{subequations}
\begin{equation}
  \label{eqdecTkIloggeq0}
  \mathcal{T}^{z}_{log}(\T, \I^{\leq}):= \textsf{AND}_{\{i:(i,j) \in \mathcal{T}(\T, \I^{\leq})\}} \mathcal{T}^{z}_{log}(\T, \I^{\leq})(i),
\end{equation}
with $\forall i \in \{i:(i,j) \in \mathcal{T}(\T, \I^{\leq})\}$:
\begin{equation}
    \label{eqdecTkIloggeq1}
    \mathcal{T}^{z}_{log}(\I, \I^{\leq})(i):= \textsf{AND}_{\{j:(i,j) \in \mathcal{T}(\I, \I^{\leq})\}} (z \leq f_{z,ij})).
      \end{equation}

     \end{subequations}

   Let us define the interval $\tau$ as follows:

   \begin{equation}
     \label{eqtauleq}
\tau= \overline{\textsf{Min}}_{\{(i,j) \in \mathcal{T}(\T,\I^{\leq})\}} f_{z,ij}(\tilde{R}^{\T}_{\tilde{\lambda} \tilde{\mu}}(x,h)).
   \end{equation}

   Then, $\tau$ is the maximal interval such that the following logical equivalence is true.
   \begin{equation}
     \label{equivZtau}
 \mathcal{T}^{z}_{log}(\T, \I^{\leq}) \mbox{ is true } \Leftrightarrow \mbox{ $\tau$ is defined by (\ref{eqtauleq})}.
     \end{equation}

   From $\tau$ the following cases may occur.

   \bit
 \item If $\overline{\textsf{argMin}}(\tau)=\{(i^{*},j^{*})\}$ then the
   set of substituable variables $S_{\tau}^{\leq}$ (which also contains the row
     where the variable appears in the constraints system) is defined by:
   \begin{equation}
          \label{Stauleq}
S_{\tau}^{\leq}:=\overline{\textsf{argMin}}(\tau),
     \end{equation}
   which is of course a singleton.
   
 \item Otherwise $\textsf{n}(\overline{\textsf{argMin}}(\tau)) \geq 2$ then let us define $\T'$ as follows:

   \begin{equation}
  \label{eqdefTT'leq}
  \T':= \left\{\bar{ll} \tilde{\B} & \mbox{ if $\exists (i,j) \in \overline{\textsf{argMin}}(\tau)$ s.t. $f^{\leq}_{ij} \in h\tilde{\B}$} \\
  \tilde{\U} & \mbox{ otherwise.}
  \ear\right.
  \end{equation}
Then, let us define the following set:

\begin{equation}
  \label{mathcalTprimeleq}
  \mathcal{T}'(\T',\overline{\textsf{argMin}}(\tau)):=\{(i,j) \in \overline{\textsf{argMin}}(\tau): f^{\leq}_{ij} \in h\T' \}.
  \end{equation}

Let us define the logical decomposition of the set $\mathcal{T}'(\T',\overline{\textsf{argMin}}(\tau))$ as follows.

     \begin{subequations}
\begin{equation}
  \label{eqdecTkprimeIloggeq0}
  \mathcal{T}'_{log}(\T',\overline{\textsf{argMin}}(\tau)):= \textsf{AND}_{\{i:(i,j) \in \mathcal{T}'(\T',\overline{\textsf{argMin}}(\tau))\}} \mathcal{T}'_{log}(\T',\overline{\textsf{argMin}}(\tau))(i),
\end{equation}
with $\forall i \in \{i:(i,j) \in (\T',\overline{\textsf{argMin}}(\tau))\}$:
\begin{equation}
    \label{eqdecTkprimeIloggeq1}
    \mathcal{T}'_{log}(\T',\overline{\textsf{argMin}}(\tau))(i):= \textsf{AND}_{\{j:(i,j) \in \mathcal{T}'(\T',\overline{\textsf{argMin}}(\tau))\}} (x_{j} \leq f^{\leq}_{ij})).
      \end{equation}

     \end{subequations}

   Then, let us defined the interval $\tau'$ as follows:

   \begin{equation}
    \label{eqtauprime}
    \tau'= \overline{\textsf{Min}}_{\{(i,j) \in  \mathcal{T}'(\T',\textsf{argMin}(\tau))\}}
    f^{\leq}_{ij}( \tilde{R}^{\T'}_{\tilde{\lambda} \tilde{\mu}}(x,h)).
   \end{equation}

Then, $\tau'$ is the maximal interval such that the following equivalence is true.

   \begin{equation}
     \label{equivXtautauprime}
     \mbox{$\mathcal{T}'_{log}(\T',\overline{\textsf{argMin}}(\tau))$ is true } \Leftrightarrow
     \mbox{ $\tau'$ is defined by
       (\ref{eqtauprime})}.
   \end{equation}
   
   If $S_{\tau'}^{\leq}$ denotes the set of subsituable variables (which also contains the row
     where the variable appears in the constraints system) then we have:

   \begin{equation}
     \label{Stauprimeleq}
S_{\tau'}^{\leq}:= \overline{\textsf{argMin}}(\tau').
     \end{equation}

   \eit

\end{theo}
\proof To prove the equivalence (\ref{equivZtau}) we just have to remark that by associativity of
$\textsf{AND}$:

\[
\bar{lll}
\mathcal{T}^{z}_{log}(\T, \I^{\leq}) & = & \textsf{AND}_{\{i:(i,j) \in \mathcal{T}(\T, \I^{\leq})\}} \mathcal{T}^{z}_{log}(\T, \I^{\leq})(i) \\
\mbox{} & = & \textsf{AND}_{\{i:(i,j) \in \mathcal{T}(\T, \I^{\leq})\}}\textsf{AND}_{\{j:(i,j) \in \mathcal{T}(\I, \I^{\leq})\}} (z \leq f_{z,ij})) \\
\mbox{} & = & \textsf{AND}_{\{(i,j) \in \mathcal{T}(\T,\I^{\leq})\}} (z \leq f_{z,ij})).
\ear
\]
The expression of the interval $\tau$ (see (\ref{eqtauleq})) is obtained by an obvious
generalization of the result of the Lemma~\ref{lemmInegalitesgeqgeq} because $\overline{\textsf{Min}}$
is associative. \\

Let us prove the logical equivalence (\ref{equivXtautauprime}) when $\textsf{n}(\textsf{argMin}(\tau)) \geq 2$. \\
In this case $\forall i \in \{i: (i,j) \in \mathcal{T}'(\T',\textsf{argMin}(\tau))\}$
all the variables $x_{j}$, $j \in \{j: (i,j) \in \mathcal{T}'(\T',\textsf{argMin}(\tau))\}$
can be exchanged. It means that any inequality $(x_{j} \leq f^{\leq}_{ij}))$ can be replaced by
a generic inequality $(y \leq f^{\leq}_{ij}))$ where $y$ is a generic variable which stands
for any $x_{j}$, $j \in \{j: (i,j) \in \mathcal{T}'(\T',\textsf{argMin}(\tau))\}$.

Now, we just have to remark that by associativity of
$\textsf{AND}$:

\[
\bar{lll}
\mathcal{T}'_{log}(\T',\overline{\textsf{argMin}}(\tau)) & = & \textsf{AND}_{\{i:(i,j) \in \mathcal{T}'(\T',\overline{\textsf{argMin}}(\tau))\}} \mathcal{T}'_{log}(\T',\overline{\textsf{argMin}}(\tau))(i) \\
\mbox{} & = & \textsf{AND}_{\{i:(i,j) \in \mathcal{T}'(\T',\overline{\textsf{argMin}}(\tau))\}} [\textsf{AND}_{\{j:(i,j) \in \mathcal{T}'(\T',\overline{\textsf{argMin}}(\tau))\}} (y \leq f^{\leq}_{ij}) ] \\
\mbox{ } & = & \textsf{AND}_{\{(i,j) \in \mathcal{T}'(\T',\overline{\textsf{argMin}}(\tau))\}} (y \leq f^{\leq}_{ij}).
\ear
\]
The expression of the interval $\tau'$ (see (\ref{eqtauprime})) is obtained by an obvious
generalization of the result of the Lemma~\ref{lemmInegalitesgeqgeq} because $\overline{\textsf{Min}}$
is associative. \\

And the proof is now achieved. \\

\cqfd

\subsection{The substitution procedure at step $k$}
\label{sublescasMax}

Assuming that the substitution process is at step $k \neq n$, we are now in position to indicate
how to choose a remaining variable to be substituted. The choice also depends on the
following different cases which are listed hereafter. \\

\noi
\textsc{Case $0$}: $\max(z)= \dzero$ or no maximum. \textsc{Case $0.1$}: $\x = \emptyset$ and the condition (\ref{eqchokmax}) of
Theorem~\ref{thmchattaignablemax} does not holds. $\x=\emptyset$ means that
$k=n$ (no more variable to be substituted) and $\textsf{cost}^{[n]}(x,h)= c_{h}^{[n]} \otimes h$. The condition (\ref{eqchokmax}) of
Theorem~\ref{thmchattaignablemax} does not holds means $A_{[|1,m|] h}^{+[n]} \nleq
A_{[|1,m|] h}^{-[n]}$ then $h=\dzero$ is the only possible value for $h$ and by backward substitution we have
$z=\dzero$, $x=\boldsymbol{-\infty}$. \\

\noi 
\textsc{Case $0.2$}:
$\x \neq \emptyset$ and $I_{\leq}(\x)=\emptyset$ and $I_{\geq}(\x)=\emptyset$ and the conditions
(\ref{eqchokmax0})-(\ref{eqchokmax}) of
Theorem~\ref{thmchattaignablemax} do not hold. The substitution process stops because no
valid inequalities can be generated. No maximum. \\

\noi
\textsc{Case $1$}: {\bf switching case}, cf. \S~\ref{subswitch}. $\x^{+}=\emptyset$ and $\textsf{cost}^{[k]}(x,h)= c_{h}^{[k]} \otimes h$ and the
condition (\ref{eqchokmax}) of Theorem~\ref{thmchattaignablemax} does not holds and $\x^{\dzero} \neq \emptyset$ and
$I_{\leq}(\x^{\dzero}) = \emptyset$ and $I_{\geq}(\x^{\dzero}) \neq \emptyset$. The substitution must switch to
the following minimizing problem:
\begin{subequations}
  \begin{equation}
\textsf{pmr}(A,b,C,d,c,c_{h},z,x,h,(m,n)),
    \end{equation}
  where the matrices $A$ and $C$ are defined by:
 \begin{equation}
    A:=A^{-[k]}_{[|1,m|][|1,n|]}, \; C:=A^{+[k]}_{[|1,m|][|1,n|]}. 
 \end{equation}
The vectors $b$ and $d$ are defined by:
 \begin{equation}
   b:= A^{-[k]}_{[|1,m|]h}, d:= A^{+[k]}_{[|1,m|]h}.
       \end{equation}
 The initial cost function of the \textsf{MPR} is characterized by the vector
 $c$ and the constant $c_{h}$ such that:
 \begin{equation}
    c:=\boldsymbol{-\infty}, \; c_{h}:=c_{h}^{[k]}. 
    \end{equation}
 And the variables of the problem are:
 \begin{equation}
    z:=z , \; x \leftarrow x^{\dzero}, \; h:=h.
 \end{equation}
Recalling that the notation $x \leftarrow x^{\dzero}$ means that we store in the $n$-dimensional vector
 $x$ the remaining variables of the set $\x^{\dzero}$. The other components (already substituted
 variables) are set to $\dzero$. \\

 And the dimensions $(m,n)$ of the \textsf{mpr} are obviously the same
 as the ones of the \textsf{MPR}.
  \end{subequations}

\noi
\textsc{Case $2$}: Not \textsc{Case $0$} and not \textsc{Case $1$}. The first case is \textsc{Case $2.1$}:
$\x=\emptyset$. This means that
$k=n$ (no more variable to be substituted) and $\textsf{cost}^{[n]}(x,h)= c_{h}^{[n]} \otimes h$. And if $A_{[|1,m|] h}^{+[n]} \leq
A_{[|1,m|] h}^{-[n]}$ then $h$ can take all possible values in particular we can take $h=\dun$ and the $\textsf{MPR}$ succeeds
with maximum $z=c_{h}^{[n]}$. Otherwise $h=\dzero$ is the only possible value for $h$ and $z=\dzero$. \\

The other case is \textsc{Case $2.2$}: $\x \neq \emptyset$ and $\exists x_{j} \in \x$:
$I_{\leq}(x_{j}) \neq \emptyset$. Let us distinguish the following subcases which are based on the form
of the cost function:
\begin{subequations}
\begin{equation}
  \textsf{cost}^{[k]}(x,h)= c^{\intercal}_{\x^{+}} \otimes x^{+} \oplus c^{\intercal}_{\x^{\dzero}} \otimes x^{\dzero} \oplus c_{h} \otimes h,
  \end{equation}
where:
\begin{equation}
  \label{eqpartitiondexmax}
  \mbox{$\x^{+}:=\{x_{j} \in \x: \; c_{j} > \dzero\}$ and $\x^{\dzero}:=\{x_{j} \in \x: \; c_{j} = \dzero\}$}.
  \end{equation}
And the vectors of remaining $x^{+}$ and $x^{\dzero}$ are defined by:
\begin{equation}
  \label{eqpartitiondex2max}
x^{+}:=(x_{j})_{\{j: x_{j} \in \x^{+}\}} \mbox{ and } x^{\dzero}:=(x_{j})_{\{j: x_{j} \in \x^{\dzero}\}}.
 \end{equation}
\end{subequations}

\noi
\textsc{Case $2.2.1$}: the conditions (\ref{eqchokmax0})-(\ref{eqchokmax}) of Theorem~\ref{thmchattaignablemax} hold. Then, $z=c_{h} \otimes h$ is
the maximum for the $\textsf{MPR}$. And the vector of the remaining variables $x$ is set to $\boldsymbol{-\infty}$. \\

\noi
\textsc{Case $2.2.2$}: otherwise, the conditions (\ref{eqchokmax0})-(\ref{eqchokmax}) of Theorem~\ref{thmchattaignablemax}
do not hold. \\

\noi
Define the set of dominating variables (recall): 

\begin{equation}
  \label{eqdomVarmax}
D:=\{x_{j} \in \x: \; I_{\geq}(x_{j}) \neq \emptyset \mbox{ and } I_{\leq}(x_{j}) \neq \emptyset \},
  \end{equation}

and the set of inequalities:

\begin{equation}
  \label{eqdefIDmax}
  \I^{\leq}:=\left\{\bar{ll} I_{\leq}(\x \cap D) & \mbox{ if $D \neq \emptyset$} \\
  I_{\leq}(\x), \mbox{ defined by (\ref{eqtteslesgeqmax})}               & \mbox{ otherwise.}
  \ear \right.
\end{equation}
Where: $I_{\leq}(\x \cap D):= \{(i,j): \mbox{ $x_{j} \in \x \cap D$ and $I_{\leq}(x_{j},i)$ is valid} \}$. \\

\noi
We apply Theorem~\ref{theoNTZminmax} with the set of valid inequalities $I^{\leq}$ defined by (\ref{eqdefIDmax})

\noi
We get a $2$-tuple $(i^{*},j^{*})$ for the substitution
of $x_{j^{*}}=f^{\leq}_{i^{*}j^{*}}(x,h)$. \\

\noi
We apply results of \S~\ref{newcharactmax} to compute the new set
of equalities $\mathcal{L}^{[k+1]}$, the new cost function $\text{cost}^{[k+1]}(x,h)$ and the new cone
 $\mathcal{C}(A^{+},A^{-})^{[k+1]}$.

\subsection{The new characteristic elements of the \textsf{MPR} after substitution}
\label{newcharactmax}

The new set of equalities is defined as follows. If conditions of
  Theorem~\ref{thmchattaignablemax} are verified then set $k+1=n$ and
  $\mathcal{L}^{[n]}:= \mathcal{L}^{[k]} \cup_{x_{j} \in \x}\{x_{j}=
  \dzero\} \cup \{z=c_{h} \otimes h\}$. Then, $\textsf{MPR}$ {\bf
    stops} and {\bf we solve} the linear system $\mathcal{L}^{[n]}$
  and express all the $x_{j}$'s as function of the homogenization
  variable $h$. Otherwise, we define $\mathcal{L}^{[k+1]}$ as:
  $\mathcal{L}^{[k+1]}:= \mathcal{L}^{[k]} \cup
   \{x_{j^{*}}=f^{\leq}_{i^{*}j^{*}}(x,h)\}$ and we apply
  Theorem~\ref{theoMinCimax}. And based on the
  previous cases the new cost function is defined as:

  \begin{equation}
  \textsf{cost}^{[k+1]}(x,h):= f_{z,i^{*}j^{*}}(x,h).
  \end{equation}
  
And in all cases the new
cone $\mathcal{C}(A^{+}, A^{-})^{[k+1]}$ associated with $\textsf{MPR}$ is deduced from the
following set of inequalities:

  \begin{subequations}
    \begin{equation}
      \bar{l}
      z \leq \textsf{cost}^{[k+1]}(x,h), \\
      \forall i \in [|1,m|]: \; a^{+[k]}_{i,.} \otimes \ell \leq  a^{-[k]}_{i,.} \otimes \ell .
      \ear
 \end{equation}
    and $\ell$ is the $n+2$-dimensional column vector which has the same components
    as the vector $w^{[k]}$ except its $j^{*}$th component wich is:
    \begin{equation}
\ell_{j^{*}}:= f^{\leq}_{i^{*}j^{*}}(x,h).
      \end{equation}
  \end{subequations}

  Let us define the following $(n+2) \times (n+2)$ {\em transition matrix} $T^{k \rightarrow k+1}$ by:

  \begin{equation}
    \forall j: t^{k \rightarrow k+1}_{j,.}:= \left\{ \bar{lc} e^{\intercal}_{j}  \mbox{ if $j \neq j^{*}$} \\
                                                (\dzero, v^{\intercal}_{i^{*}j^{*}}, r_{i^{*}j^{*}}) & \mbox{ if $j=j^{*}$.}
    \ear \right.
  \end{equation}
  Where $e^{\intercal}_{j}$ denotes the $j$th $n+2$-dimensional row vector of the
  $(n+2) \times (n+2)$-identity matrix $\bI_{n+2}$.

  The new matrices $(A^{+}, A^{-})^{[k+1]}$ are defined by:
  \begin{equation}
(A^{+}, A^{-})^{[k+1]}:= \overline{\textsf{setrowtozero}}(A^{+[k]} \otimes T^{k \rightarrow k+1}, A^{-[k]} \otimes T^{k \rightarrow k+1})
    \end{equation}
recalling that function $\overline{\textsf{setrowtozero}}$ is defined by (\ref{ligneazeromax}).

And the new vector $w^{[k+1]}$ is defined by:
\begin{equation}
w^{[k+1]}:= \left(\bar{c}
     z \\
     \vdots \\
     \dzero \\
     \vdots \\
     h
     \ear \right) \; \bar{c}
     \mbox{} \\
     \mbox{} \\
     \leftarrow \; j^{*} \\
     \mbox{} \\
     \mbox{}
     \ear 
  \end{equation}
Finally, the new set of the remaining variables at step $k+1$ is defined by:
\begin{equation}
\x := \x \setminus \{x_{j^{*}}\}.
  \end{equation}
  
\subsection{A numerical example}
\label{secExNummax}

In this section we illustrate our strongly polynomial method on a numerical
example borrowed from the litterature which are solved using pseudo-polynomial method in \cite{kn:GKS012}.
And we apply substitution on \cite{kn:GKS012}, Example $1$ p. 1458 which is:

\[
\max\{z=1 \otimes x_{1} \oplus 3 \otimes x_{2}, \; x=(x_{1}, x_{2})^{\intercal} \in \mathcal{P}(A,b,C,d) \cup \{\boldsymbol{\pm \infty}\}\},
\]
where:
\[
A:= \left(\bar{cc} \dzero & 1^{\otimes (-1)} \\
2^{\otimes (-1)} & 2^{\otimes (-1)} \\
1^{\otimes (-1)} & \dzero \\
\dun & \dzero
\ear\right), \; b:= \boldsymbol{-\infty}, \;
C:= \left(\bar{cc} \dun & \dzero \\
\dzero & \dzero  \\
\dzero & \dun \\
\dzero & 2
\ear\right), \; d:= \left(\bar{c} \dun \\ \dun \\ \dun \\ \dun \ear \right).
\]
With the following change of notations $c \leftrightarrow b$ and $B \leftrightarrow C$.

The \textsf{MPR} is defined on the homogenized cone $\mathcal{C}(A,b,C,d)$ by adding the homogenization variable $h$.

At step $0$ of the substitution the cone associated with the above constraints is
denoted $\mathcal{C}(A^{+}, A^{-})^{[0]}$ and is characterized by the following system of inequalities:

\begin{subequations}
  \begin{equation}
    \label{eqA0max}
A^{+[0]} \otimes w^{[0]} \leq A^{-[0]} \otimes w^{[0]},
    \end{equation}
  with $w^{[0]}=(z,x_{1},x_{2},h)^{\intercal} \in \overline{\real}^{4}_{\dzero}$ and the matrices $(A^{+}, A^{-})^{[0]}$ are defined
  by:

  \begin{equation}
    \label{eqA00max}
    A^{+[0]}:= \left(\bar{cccc}
\dun & \dzero & \dzero & \dzero \\
\dzero & \dzero & 1^{\otimes (-1)} & \dzero \\
\dzero & 2^{\otimes (-1)} & 2^{\otimes (-1)} & \dzero \\
\dzero &  1^{\otimes (-1)} & \dzero & \dzero \\
\dzero &  \dun & \dzero & \dzero
\ear \right), \;
A^{-[0]}:= \left(\bar{cccc}
\dzero & 1 & 3 & \dzero \\
\dzero & \dun & \dzero & \dun \\
\dzero & \dzero & \dzero &  \dun \\
\dzero & \dzero & \dun &  \dun \\
\dzero & \dzero & 2 &  \dun 
\ear \right).
    \end{equation}

\end{subequations}

And the Fourier's trick for this example is:
\begin{subequations}
\begin{equation}
\max(z)
  \end{equation}

\begin{equation}
z \leq \textsf{cost}^{[0]}(x,h)
  \end{equation}

\begin{equation}
\textsf{cost}^{[0]}(x,h)= 1 \otimes x_{1} \oplus 3 \otimes x_{2} \oplus \dzero \otimes h.
  \end{equation}  
  \end{subequations}

Here: $\x=\{x_{1}, x_{2}\}$. It is easy to check that $\forall x_{j} \in \x: I_{\leq}(x_{j}) \cup I_{\geq}(x_{j}) \neq \emptyset$. So that
the \textsf{MPR} is bounded. We have $\x^{+}=\{x_{1}, x_{2}\}$ and $\x^{\dzero}=\emptyset$. The set of dominating variables is
$D=\{x_{1}, x_{2}\}$ (see (\ref{eqdomVarmax})). The conditions (\ref{eqchokmax0})-(\ref{eqchokmax}) of Theorem~\ref{thmchattaignablemax} are not verified. \\

\noi
\textsc{Case $2.2.2$}, \S~\ref{sublescasMax} applies. We have (see (\ref{eqdefIDmax})):
$\I^{\leq}= I_{\leq}(x_{1}) \cup I_{\leq}(x_{2})=\{(1,2), (2,1), (2,2), (3,1), (4,1) \}$. \\

\noi
We apply Theorem~\ref{theoNTZminmax} with $\I^{\leq}=\{(1,2), (2,1), (2,2), (3,1), (4,1) \}$. \\

From all the possibilities listed in the following array:

\[
\bar{lllllll}
 \left( \bar{c} z \\ x_{2}
\ear \right)_{1} & \leq & \left( \bar{c}  4 \otimes x_{1} \oplus 4 \otimes h\\
1 \otimes x_{1} \oplus 1 \otimes h \ear \right) & :: & \left(\bar{c}
\mbox{$[\dzero, 4 \otimes \tilde{\lambda}]$} \\ \mbox{$[\dzero, 1 \otimes \tilde{\lambda}]$} \ear \right) & :: & \left(\bar{c}
h\tilde{\U} \\  h\tilde{\U} \ear \right) \\
 \left( \bar{c} z \\ x_{1}
\ear \right)_{2} & \leq & \left( \bar{c}  3 \otimes x_{2} \oplus 3 \otimes h\\
2 \otimes h \ear \right) & :: & \left(\bar{c}
\mbox{$[\dzero, 3 \otimes \tilde{\lambda}]$} \\ \mbox{$[\dzero, 2 \otimes \tilde{\mu}]$} \ear \right) & :: & \left(\bar{c}
h\tilde{\U} \\  h\tilde{\B} \ear \right) \\
 \left( \bar{c} z \\ x_{2}
\ear \right)_{2} & \leq & \left( \bar{c}  1 \otimes x_{2} \oplus 5 \otimes h\\
2 \otimes h \ear \right) & :: & \left(\bar{c}
\mbox{$[\dzero, 1 \otimes \tilde{\lambda}]$} \\ \mbox{$[\dzero, 2 \otimes \tilde{\mu}]$} \ear \right) & :: & \left(\bar{c}
h\tilde{\U} \\  h\tilde{\B} \ear \right) \\
 \left( \bar{c} z \\ x_{1}
\ear \right)_{3} & \leq & \left( \bar{c}  3 \otimes x_{2} \oplus 2 \otimes h\\
1 \otimes x_{2} \oplus 1 \otimes h \ear \right) & :: & \left(\bar{c}
\mbox{$[\dzero, 3 \otimes \tilde{\lambda}]$} \\ \mbox{$[\dzero, 1 \otimes \tilde{\lambda}]$} \ear \right) & :: & \left(\bar{c}
h\tilde{\U} \\  h\tilde{\U} \ear \right) \\
 \left( \bar{c} z \\ x_{1}
\ear \right)_{4} & \leq & \left( \bar{c}  3 \otimes x_{2} \oplus 1 \otimes h\\
2 \otimes x_{2} \oplus  h \ear \right) & :: & \left(\bar{c}
\mbox{$[\dzero, 3 \otimes \tilde{\lambda}]$} \\ \mbox{$[\dzero, 2 \otimes \tilde{\lambda}]$} \ear \right) & :: & \left(\bar{c}
h\tilde{\U} \\  h\tilde{\U} \ear \right).
\ear
\]
we have $\T=\tilde{\U}$ (see (\ref{eqdefTleq}),  $\mathcal{T}(\tilde{\U},\I^{\leq})= \{(1,2), (2,1), (2,2), (3,1), (4,1) \}$
(see (\ref{mathcalTleq}), $\overline{\textsf{argMin}}(\tau)=\{(2,2)\}$ with $\tau=[\dzero, 1 \otimes \tilde{\lambda}]$
defined by (\ref{eqtauleq}). The set $\overline{\textsf{argMin}}(\tau)$ is a singleton thus we have: \\
$(i^{*},j^{*})=(2,2)$ and $x_{2}=2 \otimes h$. \\

Applying the update procedure of \S~\ref{newcharactmax} the new characteristic elements are listed below. \\

The set of linear equalities is defined by:

\begin{equation}
\mathcal{L}^{[1]}=\{x_{2}=2 \otimes h\}.
\end{equation}

The new cost function is:
\begin{equation}
\textsf{cost}^{[1]}(x,h):=f_{z,22}(x,h)= 1 \otimes x_{1} \oplus 5 \otimes h.
  \end{equation}

The $4 \times 4$-transition matrix is defined by:

\begin{equation}
  T^{0 \rightarrow 1}:= \left(\bar{cccc}
  \dun & \dzero & \dzero & \dzero \\
  \dzero & \dun & \dzero & \dzero \\
  \dzero & \dzero & \dzero & 2 \\
  \dzero & \dzero & \dzero & \dun
  \ear \right). 
  \end{equation}

Then, we compute the following matrices $A^{+}:=A^{+[0]} \otimes T^{0
  \rightarrow 1}$ and $A^{-}:=A^{-[0]} \otimes T^{0 \rightarrow 1}$
and we obtain:

\[
A^{+}:= \left(\bar{cccc}
\dun & \dzero & \dzero & \dzero \\
\dzero & \dzero & \dzero & 1 \\
\dzero &  2^{\otimes (-1)} &   \dzero & \dun\\
\dzero &  1^{\otimes (-1)}  & \dzero & \dzero\\
\dzero &  \dun & \dzero & \dun
\ear \right), \;
A^{-}:= \left(\bar{cccc}
\dzero & 1 & \dzero & 5 \\
\dzero & \dun & \dzero &  \dun \\
\dzero & \dzero & \dzero &  \dun \\
\dzero & \dzero & \dzero &  2 \\
\dzero & \dzero & \dzero &  4
\ear \right).
\]
The new cone $\mathcal{C}(A^{+}, A^{-})^{[1]}$ is defined by:
\begin{subequations}
  \begin{equation}
    \label{eqA1max}
A^{+[1]} \otimes w^{[1]} \leq A^{-[1]} \otimes w^{[1]},
    \end{equation}
with $w^{[1]}= (z , x_{1} , \dzero , h)^{\intercal} \in \overline{\real}_{\dzero}^{4}$ and the matrices $(A^{\; +}, A^{ \; -})^{[1]}$
are defined as a result of $\overline{\textsf{setrowtozero}}(A^{+}, A^{-})$:

\begin{equation}
  \label{eqA11max}
A^{+[1]}:= \left(\bar{cccc}
\dun & \dzero & \dzero & \dzero \\
\dzero & \dzero & \dzero & 1 \\
\dzero &  2^{\otimes (-1)} &   \dzero & \dun\\
\dzero &  1^{\otimes (-1)}  & \dzero & \dzero\\
\dzero &  \dun & \dzero & \dun
\ear \right), \;
A^{-[1]}:= \left(\bar{cccc}
\dzero & 1 & \dzero & 5 \\
\dzero & \dun & \dzero &  \dun \\
\dzero & \dzero & \dzero &  \dun \\
\dzero & \dzero & \dzero &  2 \\
\dzero & \dzero & \dzero &  4
\ear \right).
\end{equation}

\end{subequations}

The set $\x=\{x_{1}\} \neq \emptyset$. It is easy to see that the \textsf{MPR} problem is bounded. We have
$\x^{+}=\{x_{1}\}$ and $\x^{\dzero}= \emptyset$. The set of dominating variables is $D:=\{x_{1}\}$ (see (\ref{eqdomVarmax})) .
The conditions (\ref{eqchokmax0})-(\ref{eqchokmax}) of Theorem~\ref{thmchattaignablemax} are not verified. \\

\textsc{Case $2.2.2$}, \S~\ref{sublescasMax} applies. We have (see (\ref{eqdefIDmax})):
$\I^{\leq}= I_{\leq}(x_{1})=\{(2,1), (3,1), (4,1) \}$. \\

\noi
We apply Theorem~\ref{theoNTZminmax} with $\I^{\leq}=\{(2,1), (3,1), (4,1) \}$. \\

From all the possible inequalities listed in the following array:

\[
\bar{lllllll}
\left( \bar{c} z \\ x_{1}
\ear \right)_{2} & \leq & \left( \bar{c}  5 \otimes h\\
2 \otimes h \ear \right) & :: & \left(\bar{c}
\mbox{$[\dzero, 5 \otimes \tilde{\mu}]$} \\ \mbox{$[\dzero, 2 \otimes \tilde{\mu}]$} \ear \right) & :: & \left(\bar{c}
h\tilde{\B} \\  h\tilde{\B} \ear \right) \\
\left( \bar{c} z \\ x_{1}
\ear \right)_{3} & \leq & \left( \bar{c}  5 \otimes h\\
3 \otimes h \ear \right) & :: & \left(\bar{c}
\mbox{$[\dzero, 5 \otimes \tilde{\mu}]$} \\ \mbox{$[\dzero, 3 \otimes \tilde{\mu}]$} \ear \right) & :: & \left(\bar{c}
h\tilde{\B} \\  h\tilde{\B} \ear \right)\\
\left( \bar{c} z \\ x_{1}
\ear \right)_{4} & \leq & \left( \bar{c}  5 \otimes h\\
4 \otimes h \ear \right) & :: & \left(\bar{c}
\mbox{$[\dzero, 5 \otimes \tilde{\mu}]$} \\ \mbox{$[\dzero, 4 \otimes \tilde{\mu}]$} \ear \right) & :: & \left(\bar{c}
h\tilde{\B} \\  h\tilde{\B} \ear \right).
\ear
\]
we have $\T=\tilde{\B}$ (see (\ref{eqdefTleq})), $\mathcal{T}(\tilde{\B},\I^{\leq})=\{(2,1), (3,1), (4,1) \}$
(see (\ref{mathcalTleq})), $\overline{\textsf{argMin}}(\tau)=\{(2,1), (3,1), (4,1)\}$ with
$\tau=[\dzero, 5 \otimes \tilde{\mu}]$ defined by (\ref{eqtauleq}). The set $\overline{\textsf{argMin}}(\tau)$ is
not a singleton so we have $\T' = \tilde{\B}$ (see (\ref{eqdefTT'leq})), $ \mathcal{T}'(\tilde{\B},
\overline{\textsf{argMin}}(\tau))=\{(2,1), (3,1), (4,1) \}$ (see (\ref{mathcalTprimeleq})),
$\overline{\textsf{argMin}}(\tau')=\{(2,1)\}$ with $\tau'= [\dzero, 2 \otimes \tilde{\mu}]$ defined
by (\ref{eqtauprime}). Thus, we take: \\
$(i^{*}, j^{*}) \in \overline{\textsf{argMin}}(\tau')=\{(2,1)\}$, ie. $(i^{*},j^{*})=(2,1)$, and $x_{1}=2 \otimes h$.

Applying the update procedure of \S~\ref{newcharactmax} the new characteristic elments of the problem are listed below.

The set of linear equalities is defined by:

\begin{equation}
\mathcal{L}^{[2]}=\mathcal{L}^{[1]} \cup \{x_{1}=2 \otimes h\} =\{x_{2}=2 \otimes h, x_{1}=2 \otimes h\}.
\end{equation}

The new cost function is:
\begin{equation}
\textsf{cost}^{[2]}(x,h):=f_{z,21}(x,h)= 5 \otimes h.
  \end{equation}

The $4 \times 4$-transition matrix is defined by:

\begin{equation}
  T^{1 \rightarrow 2}:= \left(\bar{cccc}
  \dun & \dzero & \dzero & \dzero \\
  \dzero & \dzero & \dzero & 2 \\
  \dzero & \dzero & \dun & \dzero \\
  \dzero & \dzero & \dzero & \dun
  \ear \right). 
  \end{equation}

We compute the following matrices $A^{+}:=A^{+[1]} \otimes T^{1
  \rightarrow 2}$ and $A^{-}:=A^{-[1]} \otimes T^{1 \rightarrow 2}$
and we obtain:

\[
A^{+}:= \left(\bar{cccc}
\dun & \dzero & \dzero & \dzero \\
\dzero & \dzero & \dzero & 1 \\
\dzero &  \dzero &   \dzero & \dun\\
\dzero &  \dzero  & \dzero & 1\\
\dzero &  \dzero & \dzero & 2
\ear \right), \;
A^{-}:= \left(\bar{cccc}
\dzero & \dzero & \dzero & 5 \\
\dzero & \dzero & \dzero & 2 \\
\dzero &  \dzero &   \dzero & \dun\\
\dzero &  \dzero  & \dzero & 2\\
\dzero &  \dzero & \dzero & 4
\ear \right).
\]
The new cone $\mathcal{C}(A^{+}, A^{-})^{[2]}$ is defined by:
\begin{subequations}
  \begin{equation}
    \label{eqA2max}
A^{+[2]} \otimes w^{[2]} \leq A^{-[2]} \otimes w^{[2]},
    \end{equation}
with $w^{[2]}= (z , \dzero , \dzero , h)^{\intercal} \in \overline{\real}_{\dzero}^{4}$ and the matrices $(A^{\; +}, A^{ \; -})^{[2]}$
are defined as a result of $\overline{\textsf{setrowtozero}}(A^{+}, A^{-})$:

 \begin{equation}
    \label{eqA22max}
A^{+[2]}:= \left(\bar{cccc}
\dun & \dzero & \dzero & \dzero \\
\dzero & \dzero & \dzero & \dzero \\
\dzero &  \dzero &   \dzero & \dzero\\
\dzero &  \dzero  & \dzero & \dzero\\
\dzero &  \dzero & \dzero & \dzero
\ear \right), \;
A^{-[2]}:= \left(\bar{cccc}
\dzero & \dzero & \dzero & 5 \\
\dzero & \dzero & \dzero & \dzero \\
\dzero &  \dzero &   \dzero & \dzero \\
\dzero &  \dzero  & \dzero & \dzero \\
\dzero &  \dzero & \dzero & \dzero
\ear \right).
\end{equation}
    
\end{subequations}

Clearly, $A_{[|1,4|]h}^{+[2]} = \boldsymbol{-\infty} \leq A_{[|1,4|]h}^{-[2]} = \boldsymbol{-\infty}$.
And the conditions (\ref{eqchokmax0})-(\ref{eqchokmax}) of Theorem~\ref{thmchattaignablemax} are verified and the cost
function $h \mapsto 5 \otimes h$ is reached.

The solution of the \textsf{MPR} is obtained by a trivial backward substitution:

\begin{equation}
z= 5 \otimes h, \; x_{2} = 2 \otimes h, \; x_{1} = 2 \otimes h.
\end{equation}

If we take $h=\dun$ then we retrieve by our strongly polynomial method the same maximum as in
Example 1 p. 1458 of \cite{kn:GKS012}. But the corner $(2,2)$ is
different from the one obtained by the pseudo-polynomial method
developed by the authors of \cite{kn:GKS012}, and mentioned p. 1474:
$x_{1}=1, x_{2}=2$.

\end{document}